%% file: main.tex
\documentclass[12pt, letterpaper]{amsart}

\usepackage{fullpage}

\usepackage{amssymb,amsthm,amsmath,amscd}

\usepackage{mathtools}
\usepackage{soul}

\RequirePackage[dvipsnames,usenames]{color}

\input{kmacros3.sty}

\usepackage{appendix}
\usepackage{tikz}
\usepackage{tikz-cd}
\usetikzlibrary{cd}
\usepackage{graphicx}
\usepackage[all,cmtip]{xy}
\usepackage{accessibility}

\usepackage{mabliautoref}
\numberwithin{equation}{theorem}

\usepackage{bm}
\usepackage{pifont}
\usepackage{upgreek}

\usepackage{eucal}
\usepackage{ulem}
\usepackage{stmaryrd}

\usepackage{setspace}

\usepackage{enumerate}
\usepackage{calc}

\usepackage{bbding}

\usepackage{verbatim}
\usepackage{alltt}

\theoremstyle{plain}
\newtheorem{maintheorem}{Main Theorem}

\renewcommand{\O}{\mathcal O}

\DeclareMathOperator{\Ass}{Ass}

\newcommand{\fm}{\mathfrak{m}}
\newcommand{\fc}{\mathfrak{c}}
\newcommand{\fb}{\mathfrak{b}}

\newcommand{\CC}{\mathbb{C}}

\newcommand{\NN}{\mathbb{N}}
\newcommand{\fa}{\mathfrak{a}}
\newcommand{\OO}{\mathcal{O}}

\DeclareMathOperator{\sep}{sep}

\DeclareMathOperator{\JJ}{\mathcal{J}}

\DeclareMathOperator{\BCM}{BCM}

\DeclareMathOperator{\Max}{Max}
\DeclareMathOperator{\Gal}{Gal}

\begin{document}

\title[Symbolic Powers and Strong $F$-regularity]{Strong $F$-regularity and the Uniform Symbolic Topology Property}
\author{Thomas Polstra}
\address{Department of Mathematics, University of Alabama, Tuscaloosa, AL 35487 USA}
\email{tmpolstra@ua.edu}
\thanks{Polstra was supported in part by NSF Grant DMS \#2101890}

\begin{abstract}
We investigate the containment problem of symbolic and ordinary powers of ideals in a commutative Noetherian domain $R$. Let $R$ be a normal domain of prime characteristic $p>0$ that is $F$-finite or essentially of finite type over an excellent local ring. Assume there exists a finite extension $R\to S$ so that the non-strongly $F$-regular locus of $\mathrm{Spec}(S)$ consists only of isolated points, then there exists a constant $C$ such that for all ideals $I \subseteq R$ and $n \in \mathbb{N}$, the symbolic power $I^{(Cn)}$ is contained in the ordinary power $I^n$. In other words, $R$ enjoys the Uniform Symbolic Topology Property.

Moreover, if $R$ is $F$-finite and strongly $F$-regular, then $R$ enjoys a property that is proven to be stronger: there exists a constant $e_0 \in \mathbb{N}$ such that for any ideal $I \subseteq R$ and all $e \in \mathbb{N}$, if $x \in R \setminus I^{[p^e]}$, then there exists an $R$-linear map $\varphi: F^{e+e_0}_*R \to R$ such that $\varphi(F^{e+e_0}_*x) \notin I$.
\end{abstract}
\maketitle

\input{SFR_intro}

\input{SFR_prelim}

\input{SFRUSTPnew}

\input{SFR_finite_Extensions}

\bibliographystyle{skalpha}
\bibliography{main}
\end{document}

%% file: SFR_intro.tex

\section{Introduction}

Throughout this article $R$ is a commutative Noetherian ring with identity. The theory of primary decompositions is foundational to commutative algebra and neighboring disciplines. Emmy Noether's unifying result is every ideal $I \subseteq R$ admits a (minimal) primary decomposition. Unlike the unique factorization of elements in unique factorization domains, primary decompositions of an ideal are not unique at embedded components, only at minimal components. If $\fp \subseteq R$ is a prime ideal and $n \in \mathbb{N}$, then the $n$th symbolic power $\fp^{(n)} := \fp^n R_\fp \cap R$ is the minimal component of $\fp^n$ appearing in every choice of a primary decomposition of $\fp^n$. The containment problem for the prime $\fp$ is the problem of establishing ideal containments $\fp^{(t)} \subseteq \fp^s$, i.e., understanding when the unique minimal component of $\fp^t$ is contained in all possible choices of primary components of $\fp^s$.

The containment problem partially originates from Hartshorne’s investigations into a non-local analogue of local duality. In \cite{hartshorneAffineDuality}, Hartshorne posed the question of when the topologies defined by the symbolic and ordinary powers of a prime $\fp \in \Spec(R)$ are \emph{equivalent}; that is, for every $s\in\NN$ there exists $t \in \mathbb{N}$ such that $\fp^{(t)} \subseteq \fp^s$. Criteria for the equivalence of the adic and symbolic topologies of a prime ideal have been established in \cite{McAdam, Schenzel, McAdamRatliffSymbolic}. Notably, if $R$ is an excellent normal domain and $\fp \subseteq R$ a prime ideal, then the symbolic and adic topologies of $\fp$ are necessarily equivalent.

More generally, for any ideal $I \subseteq R$, $W$ the complement of the union of the associated primes of $I$, and $n \in \mathbb{N}$, the \emph{$n$th symbolic power of $I$} is the ideal $I^{(n)} := I^n R_W \cap R$. The equivalence of symbolic and adic topologies of an ideal implies a deeper relationship. Swanson proved that if the symbolic and adic topologies of an ideal $I$ are equivalent, then there exists a constant $C$, depending on $I$, such that for all $n \in \mathbb{N}$, $I^{(Cn)} \subseteq I^n$, \cite[Main Result]{SwansonLinear}. Swanson’s theorem inspired further investigation into the containment problem, leading to characteristic-dependent developments that reveal a remarkable property in regular rings.

\begin{theorem}[Uniform Symbolic Topology in Non-Singular Rings]
    Let $R$ be a Noetherian non-singular ring and $I\subseteq R$ an ideal. If $h$ is the largest height of an associated prime of $I$, then for all $n \in \mathbb{N}$, 
    \[
    I^{(hn)} \subseteq I^n.
    \]    
    Outline of theorem development:
    \begin{itemize}
        \item Radical ideals of a non-singular ring essentially of finite type over $\CC$, \cite{ELS}.
        \item All ideals of a non-singular ring containing a field, \cite{HHComparison}.
        \item Radical ideals of an excellent non-singular ring of mixed characteristic, \cite{MaSchwedeSymbolic}.
        \item All ideals of a non-singular ring of mixed characteristic, \cite{MurayamaSymbolic}.
    \end{itemize}
\end{theorem}

Proof methodology that non-singular rings exhibit the above uniform comparison between the symbolic and ordinary powers of ideals have had transformative effects, requiring wide-reaching developments in multiplier and test ideal theory, closure operations, perfectoid geometry, and the construction of big Cohen-Macaulay algebras in rings of all characteristics; see \cite{HHAnnals, HaraYoshida, HaraTestIdeals, TakagiYoshida, Dietz, RG, MaSchwedeSymbolic, BhattCM, HaconLamarcheSchwede, MurayamaSymbolic}. In addition to completing the uniform comparison of ordinary and symbolic powers of all ideals in all regular rings, Murayama’s article \cite{MurayamaSymbolic} provides an excellent exposition of the connection of topics surrounding the comparison between symbolic and ordinary powers of ideals in regular rings.

With limited exception, singularities do not enjoy the same comparisons of ordinary and symbolic powers of ideals.\footnote{A notable exception of a singularity class are diagonally $F$-regular rings, see \cite{DiagonallyFRegular}.} To account for the presence of singularities, we say that a ring $R$ enjoys the \emph{Uniform Symbolic Topology Property} if there exists a constant $C$ so that for all $I\subseteq R$ and $n\in\NN$, if $h$ is the largest height of an associated prime of $I$, then $I^{(Chn)}\subseteq I^n$. \footnote{If $R$ has finite Krull dimension, then $R$ enjoys the Uniform Symbolic Topology Property if and only if there exists a constant $C$ so that for all ideals $I\subseteq R$ and $n\in\NN$, $I^{(Cn)}\subseteq I^n$.} 

\begin{question}[Uniform Symbolic Topology Problem]\label{question USTP}
    What classes of rings enjoy the Uniform Symbolic Topology Property?
\end{question}

Question~\ref{question USTP} is well-posed for an excellent domain $R$ if and only if the normalization map $R\to \overline{R}$ enjoys the going-down property, see Proposition~\ref{proposition equivalent topologies}, c.f. \cite[Proposition]{HKVErratum}. Initial progress on the Uniform Symbolic Topology Problem in singular rings and related considerations are the main focus of the articles \cite{HKV, Walker2, Walker1, Walker3, HKAbelian, HKVfinite, DiagonallyFRegular, DeStefaniGrifoJeffries, DSGJ, HKHypersurface, PolstraZariskiNagata}. Of notable significance, isolated singularities that are either essentially of finite type over a field of characteristic $0$ or of prime characteristic and $F$-finite enjoy the Uniform Symbolic Topology Property, see \cite[Main Result]{HKV}.



\subsection*{Summary of Main Results} If $R$ is a ring of prime characteristic $p>0$ then $F:R\to R$ denotes the Frobenius endomorphism and $F_*R$ the Frobenius pushforward of $R$. The singularities of a prime characteristic ring are determined by properties of the Frobenius endomorphism, a heuristic indicated by Kunz's theorem equating non-singularity of $R$ with $F_*R$ being a flat $R$-module, \cite{KunzPrimeRegular}. 

If $R$ is non-singular and of prime characteristic, then $F$ being flat lends itself to a simple proof that $R$ enjoys the Uniform Symbolic Topology Property, a result whose proof technique for prime ideals is presented in the introduction of \cite{HHComparison}. A critical detail that requires non-singularity, and not easily navigated in singular rings, is that for all ideals $I\subseteq R$ and $e\in\NN$, the associated primes of $I^{[p^e]}:=F^e(I)R$ agree with the set of associated primes of $I$. In rings with singularities, not only can the ideal $I^{[p^e]}$ have associated primes that are not associated to $I$, the collection $\bigcup_{e\in\NN}\Ass(I^{[p^e]})$ can be infinite, even if the singularities of $R$ are mild, \cite{Katzman,SinghSwanson}.

The class of strong $F$-regular singularities, brought to light by Hochster's and Huneke's theory of tight closure, \cite{HHmem}, is a central focus of prime characteristic rings with connections to singularities of the complex minimal model program and mixed characteristic singularities via reduction to positive characteristic arguments and existence of big Cohen-Macaulay algebras. A singularity over $\CC$ is klt if and only if its reduction modulo $p$ is strongly $F$-regular for almost all primes $p$, \cite[Main Result]{HaraWatanabe}, cf. \cite{SmithRational, HaraRational, MehtaSrinivas}. Ma and Schwede define the class of $\BCM$-singularities in \cite{MaSchwedeSingularities}, a characteristic-free notion, which in rings of prime characteristic coincides with the class of strongly $F$-regular singularities by a theorem of Schwede and Smith; see \cite[Definition-Proposition~2.7]{MaSchwedeSingularities} and \cite[Theorem~4.3 and Remark~4.10]{SchwedeSmith}.

Our main contribution to Question~\ref{question USTP} implies every $F$-finite strongly $F$-regular domain of prime characteristic $p>0$ enjoys the Uniform Symbolic Topology Property. The method of proof stems from properties of Jacobian ideals in positive characteristic described by \cite[Lemma~3.6]{HHComparison} and comparisons of ideal topologies under finite extensions found in \cite{HKVfinite}.

\begin{maintheorem}\label{Main theorem USTP in SFR rings}
    Let $R$ be a normal domain of prime characteristic $p>0$ that is either $F$-finite or essentially of finite type over an excellent local ring. If there exists a finite extension $R\to S$ so that $S_\fp$ is strongly $F$-regular for all non-maximal prime ideals $\fp\in\Spec(S)$, then $R$ enjoys the Uniform Symbolic Topology Property.
\end{maintheorem}

It seems reasonable to pursue proof a klt singularity over $\mathbb{C}$, or more generally that a ring of arbitrary characteristic with $\BCM$-singularities, enjoys the Uniform Symbolic Topology Property.  In characteristic $0$, Main Theorem~\ref{Main theorem USTP in SFR rings}, and a reduction to positive characteristic argument, is currently insufficient to assert a klt singularity enjoys the Uniform Symbolic Topology Property. Unlike reduction to positive characteristic arguments of \cite{HHComparison} that generalize characteristic $0$ results of \cite{ELS}, much of the uniformity in the proof of Main Theorem~\ref{Main theorem USTP in SFR rings} depends on the characteristic of $R$ and Artin-Rees numbers. Similar issues are present in recent progress on the Uniform Symbolic Topology Problem in \cite{HKHypersurface} and in the study of numerical measurements of singularities intrinsic to the Frobenius endomorphism. See \cite{CaminataShidelerTuckerZerman} for recent progress in this direction. The pursuit of a proof of Main Theorem~\ref{Main theorem USTP in SFR rings} for mixed characteristic rings with $\BCM$-singularities encounters an additional challenge, a suitable mixed characteristic analogue of \cite[Lemma~3.6]{HHComparison} has yet to be established, see Question~\ref{question symbolic multipliers exist}.

The proof of Main Theorem~\ref{Main theorem USTP in SFR rings} begins as a study of $R$-linear maps $F^e_*R\to R$ of an $F$-finite domain $R$ through the lens of splitting ideals. If $I\subseteq R$ is an ideal and $e\in\NN$, then the \emph{$e$th splitting ideal of $I$} is
\[
I_e(I):=\{r\in R\mid \forall\varphi\in\Hom_R(F^e_*R,R),\, \varphi(F^e_*r)\in I\}.
\]
Splitting ideals became objects of interest in the study of the $F$-signature of a local ring, see \cite{SmithVanDenBergh, HunekeLeuschke} for the origins of $F$-signature theory and \cite{AberbachEnescu,YaoObservations,TuckerFsigExists} for the use of splitting ideals in the study of $F$-signature. In summary, if $(R,\fm,k)$ is a local $F$-finite domain of prime characteristic $p>0$, then the quotients $F^e_*R/F^e_*I_e(\fm)$ are annihilated by $\fm$, $\dim_k(F^e_*R/F^e_*I_e(\fm))$ measures the maximal rank of a free $R$-summand of $F^e_*R$, and the \emph{$F$-signature of $R$} is 
\[
s(R):=\lim_{e\to \infty} \frac{\dim_k(F^e_*R/F^e_*I_e(\fm))}{\rank_R(F^e_*R)},
\]
a well-defined number by \cite[Main Result]{TuckerFsigExists}. The $F$-signature is a fundamental measurement of singularities as a local $F$-finite ring $(R,\fm,k)$ is strongly $F$-regular if and only if $s(R)>0$ by \cite[Main Result]{AberbachLeuschke}. The critical detail of proof of \cite[Main Result]{AberbachLeuschke} being that if $(R,\fm,k)$ is strongly $F$-regular then there exists a constant $e_0$ so that for all $e\in\NN$, $I_{e+e_0}(\fm)\subseteq \fm^{[p^e]}$. We give a uniform extension of their theorem.

\begin{maintheorem}
    \label{Main Theorem USTP for splitting ideals}
    Let $R$ be an $F$-finite strongly $F$-regular domain of prime characteristic $p>0$. There exists a constant $e_0\in\NN$ so that for all ideals $I\subseteq R$ and $e\in\NN$, $I_{e+e_0}(I)\subseteq I^{[p^e]}.$ Equivalently, if $x\in R\setminus I^{[p^e]}$, then there exists $\varphi\in\Hom_R(F^{e+e_0}_*R,R)$ so that $\varphi(F^{e+e_0}_*x)\not\in I$.
\end{maintheorem}

\begin{remark}
    Hypotheses milder than those of Main Theorem~\ref{Main Theorem USTP for splitting ideals} imply that for all $e\in\NN$, $\Ass(I_{e}(I))=\Ass(I)$, see Lemma~\ref{lemma Some properties of splitting ideals} (\ref{associated primes of splitting ideals}). On the other hand, $\bigcup_{e\in\NN}\Ass(I^{[p^e]})$ could be an infinite set,  \cite[Theorem~2.1]{Katzman}, even if $R$ is an $F$-finite strongly $F$-regular unique factorization hypersurface domain, \cite[Theorem~1.2]{SinghSwanson}.
\end{remark}

\begin{remark}
    Let $R$ be an $F$-finite strongly $F$-regular domain of prime characteristic $p>0$. The property that there exists $e_0\in\NN$ so that for all ideals $I\subseteq R$ and $e\in\NN$, $I_{e+e_0}(I)\subseteq I^{[p^e]}$ is a property that implies $R$ enjoys the Uniform Symbolic Topology Property, see Lemma~\ref{lemma method of attack to USTP}.
\end{remark}

\begin{remark}    
    While not a direct statement about test ideals, Main Theorem~\ref{Main Theorem USTP for splitting ideals} concerns foundational constructions from which prime characteristic test ideal theory are derived, thus offering deeper insight into the test ideal theory of prime characteristic rings. For additional details on the connections between test ideals and splitting ideals, see (\ref{Section connections with test ideals}). 
\end{remark}

\subsection*{Paper Organization} Section~\ref{Section Preliminary} organizes preliminary results on $F$-finite rings, linear comparison of ideal topologies, splitting ideals, a background of uniform properties of Noetherian rings, and the theory of uniform symbolic multipliers. Section~\ref{Section Groundwork on powers and splitting ideals} significantly builds upon the prime characteristic theory presented in the preliminary section. Theorem~\ref{Theorem Splitting ideals and powers of ideals} is an interesting insight into splitting ideals as it is applicable to any $F$-finite domain.

\begin{maintheorem}[Theorem~\ref{Theorem Splitting ideals and powers of ideals}]
    Let $R$ be an $F$-finite domain of prime characteristic $p > 0$. There exists a constant $C$ such that for all ideals $I \subseteq R$ and $e \in \mathbb{N}$,
    \[
    I_e(I^C) \subseteq I^{[p^e]}.
    \]
\end{maintheorem}

In an effort to make the novel techniques underlying the main results more transparent, Section~\ref{Section Groundwork on powers and splitting ideals} concludes with a proof of Main Theorems~\ref{Main theorem USTP in SFR rings} and~\ref{Main Theorem USTP for splitting ideals} under simplified hypotheses, but presents a case of the first set of main theorems that are not covered by prior progress on the Uniform Symbolic Topology Problem; see Proposition~\ref{proposition ustp sfr restricted assumptions}. Section~\ref{Section Main Theorems} features Theorem~\ref{theorem USTP in isolated non SFR rings}, a theorem that implies Main Theorem~\ref{Main Theorem USTP for splitting ideals} and Main Theorem~\ref{Main theorem USTP in SFR rings} when $R=S$.

To capture the full statement of Main Theorem~\ref{Main theorem USTP in SFR rings} we extend results of \cite[Section~3]{HKVfinite}, results only applicable to prime ideals, to all ideals. Main Theorem~\ref{Main theorem USTP in SFR rings} then follows as a corollary of Theorem~\ref{theorem USTP in isolated non SFR rings}, Main Theorem~\ref{Main Theorem USTP finite}, and some additional materials found in Section~\ref{Section Symbolic Powers, Intersections, finite extensions}.

\begin{maintheorem}
    \label{Main Theorem USTP finite}
    Let $R$ be a normal domain and assume that $R$ satisfies one of the following hypotheses:
    \begin{itemize}
        \item $R$ is essentially of finite type over an excellent local ring;
        \item $R$ is of prime characteristic and $F$-finite;
        \item $R$ is essentially of finite type over $\ZZ$.
    \end{itemize}
    Let $R\to S$ be a finite normal domain extension of $R$ that enjoys the Uniform Symbolic Topology Property. Then $R$ enjoys the Uniform Symbolic Topology Property.
\end{maintheorem}

\section*{Acknowledgments}

The author thanks Daniel Katz for numerous valuable discussions during the preparation of this manuscript. The author thanks Rankeya Datta for feedback on a preliminary version of the article that led to significant improvements of the main results.

%% file: SFR_prelim.tex
\section{Preliminary Results}\label{Section Preliminary}
\subsection{Rings of Prime Characterisitic} Let $R$ be a Noetherian ring of prime characteristic $p>0$ and for each $e\in\NN$ let $F^e:R\to R$ denote the $e$th iterate of the Frobenius endomorphism. Let $F^e_*R$ denote the $R$-module obtained through restriction of scalars of $F^e:R\to R$. The ring $R$ is said to be \emph{$F$-finite} if $F^e_*R$ is finitely generated for some, equivalently for all, $e\in\NN$.

\begin{definition}
    \label{defintion of F-singularities}
    Let $R$ be a Noetherian ring of prime characteristic $p>0$.
    \begin{itemize}
        \item We say that $R$ is \emph{$F$-pure} if $R\to F^e_*R$ is pure\footnote{A map of $R$-modules $N\to M$ is called \emph{pure} if for $R$-modules $P$, $N\otimes_R P\to M\otimes_R P$ is injective.} for some, equivalently for all, $e\in\NN$.
        \item We say that $R$ is \emph{strongly $F$-regular} if for all elements $c\in R$ avoiding the minimal primes of $R$, there exists $e\in\NN$ so that $R\to F^e_*R\xrightarrow{\cdot F^e_*c}F^e_*R$ is pure. 
    \end{itemize}
\end{definition}

\begin{remark}
    There are naturally competing notions of strong $F$-regularity in non-$F$-finite rings. Our definition of strongly $F$-regular aligns with the notion of \emph{very strongly $F$-regular} as defined in \cite{HochsterYaoTrans}. The prime characteristic results presented in this article pertain to rings that are either $F$-finite or essentially of finite type over an excellent local ring. In these scenarios, the competing notions of strong $F$-regularity discussed in \cite{HochsterYaoTrans} coincide; see \cite[Theorem~2.23 (2)]{HochsterYaoTrans}.
\end{remark}

The following lists well-known properties of $F$-finite rings used throughout the article. Items (\ref{purity condition}) and (\ref{F regular condition}) are standard applications of \cite[Theorem~2.6]{HochsterPurity}.
\begin{theorem}[Properties of $F$-finite rings]
    \label{Kunz's Theorems on F-finite rings}
    Let $R$ be an $F$-finite ring of prime characteristic $p>0$.
    \begin{enumerate}
        \item $R$ has finite Krull dimension, \cite[Proposition~1.1]{KunzExcellent}.
        \item $R$ is excellent, \cite[Theorem~2.5]{KunzExcellent}. Consequently,
        \begin{itemize}
            \item the singular locus of $\Spec(R)$ is closed;
            \item if $R$ is reduced, then $R_\fp$ is analytically unramified for all $\fp\in\Spec(R)$;
            \item if $\fp\in\Spec(R)$, $R_\fp$ is a normal domain, then $R_\fp$ is analytically irreducible.
        \end{itemize}
        \item\label{Kunz regular} $R$ is regular if and only if for some, equivalently for all, $e\in\NN$, $F^e_*R$ is a finitely generated flat $R$-module, \cite[Theorem~2.1]{KunzPrimeRegular}.
        \item\label{purity condition} $R$ is $F$-pure if and only if for some, equivalently for all, $e\in\NN$, there exists an onto $R$-linear map $F^e_*R\to R$.
        \item\label{F regular condition} $R$ is strongly $F$-regular if and only if for all $c\in R$ avoiding the minimal primes of $R$, there exists $e\in\NN$ and an onto $R$-linear map $\varphi:F^e_*R\to R$ so that $\varphi(F^e_*c)=1$.
        \item\label{F regular is normal} If $R$ is strongly $F$-regular then $R$ is normal, \cite[Theorem~5.5 (d)]{HHTAMS}.
        \item The non-strongly $F$-regular locus of $\Spec(R)$ is closed, \cite[Corollary~10.14]{HochsterFoundations}.
    \end{enumerate}
\end{theorem}

Our study of the powers and symbolic powers of ideals is enhanced by integral closures of the powers of an ideal as well as the ``Frobenius symbolic powers'' of an ideal.

\begin{definition}
    \label{definition various powers of ideals}
    Let $R$ be a Noetherian ring, $I\subseteq R$ an ideal, $W$ the complement of the union of the associated primes of $I$, and $e,n\in\NN$. 
\begin{itemize}
    \item $\overline{I}$ is the \emph{integral closure} of $I$ and consists of elements $r\in R$ so that there exists an equation $r^t+a_1r^{t-1}+\cdots + a_{t-1}r+a_t=0$ so that $a_i\in I^i$.
    \item $I$ is \emph{integrally closed} if $I=\overline{I}$.
    \item If $R$ is of prime characteristic $p>0$, then the \emph{$e$th Frobenius symbolic power of $I$} is the ideal $I^{\left([p^e]\right)}:=I^{[p^e]}R_W\cap R$.
    \begin{itemize}
        \item If $R$ is non-singular, i.e., $F^e_*R$ is a flat $R$-module for all $e$, then $\Ass(I) = \Ass(I^{[p^e]})$. In this case, we have $I^{[p^e]} = I^{\left([p^e]\right)}$.
    \end{itemize}
\end{itemize}
\end{definition}

Integral closure of ideals benefits will benefit our study throughout this article using standard techniques involving the theory of minimal reductions, a corollary of Rees' valuation criteria for containment in the integral closure of an ideal, and the Uniform Brian\c{c}on-Skoda Theorem, see \cite[Section~8]{SwansonHuneke}, \cite[Corollary~6.8.12]{SwansonHuneke}, and \cite[Theorem~4.13]{HunekeUniformBounds} respectively. Lemma~\ref{lemma uniform symbolic Frobenius power multipliers} is a typical demonstration of such methods.

\subsection{Linear Equivalence of Ideal Topologies}\label{section linear equivalence}

Let $R$ be a Noetherian ring, $\mathbb{I}=\{I_n\}$ and $\mathbb{J}=\{J_n\}$ be descending chains of ideals of $R$. The $\mathbb{I}$-topology is \emph{finer} than the $\mathbb{J}$-topology if for all $s\in\NN$ there exists $t\in\NN$ so that $I_t\subseteq J_s$. The $\mathbb{I}$-topology is \emph{equivalent} to the $\mathbb{J}$-topology if the $\mathbb{I}$-topology is finer than the $\mathbb{J}$-topology and the $\mathbb{J}$-topology is finer than the $\mathbb{I}$-topology. If $I\subseteq R$ is an ideal, then the $I$-adic topology is determined by the descending chain of ideals $\{I^n\}$ and the $I$-symbolic topology is determined by the descending chain of ideals $\{I^{(n)}\}$.

 Equivalence of the adic and symbolic topologies of all prime ideals of an excellent domain $R$ is equivalent to the going-down property between $R$ and its normalization $\overline{R}$, cf. \cite[Proposition]{HKVErratum}. Relevant to Main Theorem~\ref{Main theorem USTP in SFR rings}, if $R$ is an excellent domain, $R_\fp$ is normal for all non-maximal primes $\fp$, then the going-down property is enjoyed by the normalization map $R\to \overline{R}$. The following proposition outlines necessary details to extend \cite[Proposition]{HKVErratum} to conclude the Uniform Symbolic Topology Problem is well-posed for all excellent domains $R$ that enjoys the going-down property with respect to its normalization.

\begin{proposition}[{Generalization of \cite[Proposition]{HKVErratum}}]
    \label{proposition equivalent topologies}
    Let $R$ be an excellent domain and $\overline{R}$ the normalization of $R$. The following are equivalent:
    \begin{enumerate}
        \item $R\to \overline{R}$ enjoys the going-down property.
        \item For every ideal $I\subseteq R$, the adic and symbolic topologies of $I$ are equivalent. 
    \end{enumerate}
\end{proposition}

\begin{proof}
    Assume the normalization map $R\to \overline{R}$ enjoys the going-down property and let $I\subseteq R$ be an ideal. Let $\{\fp_1,\fp_2,\ldots,\fp_t\}=\bigcup_{n\in\NN}\Ass(I^n)$. By a theorem of Swanson, \cite[Main Result]{SwansonPrimaryDecompositions}, there exists a constant $C$ and primary decompositions $I^n=\fp_{1,n}\cap \fp_{2,n}\cap \cdots \cap \fp_{t,n}$ so that each $\fp_{i,n}$ are $\fp_i$-primary and $\fp_{i}^{Cn}\subseteq \fp_{i,n}$.\footnote{If $\fp_i\in\bigcup_{n\in\NN}\Ass(I^n)$ but $\fp_i\not \in \Ass(I^n)$, let $\fp_{i,n}=R$.} For each $i$, there is a minimal prime $\fq$ of $I$ contained in $\fp_i$. The adic topology of $I$ is equivalent to the adic topology of $\sqrt{I}$. Therefore to show the symbolic topology of $I$ is finer than the adic topology of $I$, it suffices to show for each $\fp_i$ and minimal prime of $I$ of $\fq$ contained in $\fp_i$, the symbolic topology of $\fq$ is finer than the adic topology of $\fp_i$. This is indeed the case as the adic and symbolic topology of the prime $\fq$ are equivalent by \cite[Proposition]{HKVErratum}.

    If the going-down property is not enjoyed by $R\to \overline{R}$, then there exists a prime $\fp\in\Spec(R)$ whose adic and symbolic topologies are not equivalent, see \cite[Proposition]{HKVErratum} for relevant details.
\end{proof}


\subsection{Splitting Ideals}

Recall that if $R$ is a $F$-finite ring of prime characteristic $p>0$, $I\subseteq R$, and $e\in\NN_{\geq 1}$, then
\[
I_e(I):=\{r\in R\mid \forall \varphi\in\Hom_R(F^e_*R,R),\, \varphi(F^e_*r)\in I\}.
\]
We set $I_0(I)=I$ for all ideals $I\subseteq R$. The sets $I_e(I)$ defined above are easily checked to form ideals of $R$ and are the \emph{$e$th splitting ideal of $I$}. The below Lemma~\ref{lemma Some properties of splitting ideals} lists properties of splitting ideals that generalizes and supplements observations of \cite[Section~2]{PolstraSmirnovEquimultiplicity}. All statements and proof methods are likely well-known by experts.

\begin{lemma}[Properties of Splitting Ideals]
    \label{lemma Some properties of splitting ideals}
    Let $R$ be an $F$-finite domain of prime characteristic $p>0$, $I\subseteq R$ an ideal, $W$ the complement of the union of the associated primes of $I$, and $e\in\NN$.
    \begin{enumerate}
        \item\label{splitting ideals and localization} If $S\subseteq R$ is a multiplicative set, then $I_e(I)R_S = I_{e}(IR_S)$.
        \item\label{splitting ideals containment of bracket powers} $I^{[p^e]}\subseteq I_e(I)$.
        \item\label{splitting ideals and colons} For every $c\in R$, $I_e((I:_Rc))=(I_e(I):_Rc^{p^e})$.
        \item\label{splitting ideal containment} If $I\subseteq J$ is a containment of ideals then $I_e(I)\subseteq I_e(J)$.
        \item\label{splitting ideals strict containment} If $R$ is $F$-pure and $I\subsetneq J$ a strict containment of ideals, then $I_e(I)\subsetneq I_e(J)$.
        \item\label{associated primes of splitting ideals} If $R_W$ is $F$-pure then $\Ass(I)=\Ass(I_e(I))$.
        \item\label{splitting ideals regular} If $R$ is non-singular then $I_e(I)=I^{[p^e]}=I^{([p^e])}$.
        \item\label{splitting ideals and intersection} Given  a non-empty collection of ideals $\Lambda$, $I_e\left(\bigcap_{I\in\Lambda}I\right) = \bigcap_{I\in\Lambda}I_e(I)$.
        \item\label{splitting ideals some comparisons 1} If $R_W$ is $F$-pure then $I^{\left([p^{e}]\right)}\subseteq I_{e}(I)$.
        \item\label{splitting ideals some comparisons 2} If $e,e_1\in\NN$ then $I_{e+e_1}(I)\subseteq I_e(I_{e_1}(I))$.
        \item\label{splitting ideals some comparisons 3} If $R_W$ is $F$-pure and $e,e_1\in\NN$, then $I_{e+e_1}(I)\subseteq I_e(I)$.
        \item\label{splitting ideals multiplication containment} For every $c\in R$, $c^{p^e}I_e(I)\subseteq I_e(cI)$.
        \item\label{splitting ideals adjoin variable} Let $T$ be a variable. Then $I_e(I)R[T]=I_e(IR[T])$. 
    \end{enumerate}
\end{lemma}

\begin{proof}
    The assumption $R$ is $F$-finite, i.e., $F^e_*R$ is a finitely generated $R$-module, implies that if $S$ is a multiplicative set, then $\Hom_R(F^e_*R, R)_S \cong \Hom{R_S}(F^e_*R_S, R_S)$. Statement (\ref{splitting ideals and localization}) is then a direct consequence of $R$ being $F$-finite. Statements (\ref{splitting ideals containment of bracket powers}) and (\ref{splitting ideals and colons}) can then be checked locally at the maximal ideals of $R$ and therefore are the content of \cite[Lemma~2.2 (2)]{PolstraSmirnovEquimultiplicity} and \cite[Lemma~2.2 (7)]{PolstraSmirnovEquimultiplicity} respectively. The containment of (\ref{splitting ideal containment}) is immediate by definition. Statement (\ref{splitting ideals strict containment}) can be checked after localizing at a maximal ideal $\fm$ of $R$ so that $IR_\fm \subsetneq JR_\fm$. Therefore (\ref{splitting ideals strict containment}) follows by (\ref{splitting ideals and localization}) and \cite[Corollary~2.4]{PolstraSmirnovEquimultiplicity}.

    For (\ref{associated primes of splitting ideals}), we first note that if $\fq\not\in\Ass(I)$, then by (\ref{splitting ideals and localization}) we can check that $\fq\not\in\Ass(I_e(I))$ by replacing $R$ by the localization $R_\fq$ and assume that $(R,\fm,k)$ is local, $\fm\not\in \Ass(I)$, and $R_W$ is $F$-pure. Suppose that $c\in \fm\setminus \bigcup_{\fq\in\Ass(I)}\fq$. By (\ref{splitting ideals and colons}), $I_e(I)=I_e((I:_Rc))=(I_e(I):_Rc^{p^e})$. Therefore $\fm\not \in \Ass(I_e(I))$ provided that $I_e(I)$ is a proper ideal of $R$. This is indeed the case since $1\not \in I_e(IR_W)$, as $R_W$ is $F$-pure by assumption, and $I_e(IR_W)=I_e(I)R_W$ by (\ref{splitting ideals and localization}). Therefore $\Ass(I_{e}(I))\subseteq \Ass(I)$. Conversely, if $\fq\in\Ass(I)$, to prove $\fq\in\Ass(I_e(I))$ we can replace $R$ by $R_\fq$ by (\ref{splitting ideals and localization}), assume $(R,\fm,k)$ is local, $\fm\in \Ass(I)$, $R$ is $F$-pure, and demonstrate that $\fm\in\Ass(I_e(I))$. If $I$ is $\fm$-primary, then $I_e(I)$ is $\fm$-primary by (\ref{splitting ideals containment of bracket powers}). If $I$ is not $\fm$-primary, $\fm$ is an associated prime of $I$, and assume by way of contradiction that $\fm$ is not an associated prime of $I_e(I)$. Choose element $c\in \fm$ that avoids the non-maximal associated primes of $I$, avoids the associated primes of $I_e(I)$, and so that $I\subsetneq (I:_R\fm^{\infty})= (I:_R c)$. By (\ref{splitting ideals and colons}) and (\ref{splitting ideals strict containment}), $I_e(I)\subsetneq I_e((I:_Rc))=(I_e(I):_Rc^{p^e})$. But $c$ avoiding the associated primes of $I_e(I)$ implies $I_e(I)=(I_e(I):_Rc^{p^e})$, a contradiction.

    In Claim (\ref{splitting ideals regular}) we assume $R$ is non-singular, i.e., $F^e_*R$ is a flat $R$-module. Restriction of scalars is exact, therefore $\Ass(I^{[p^e]})=\Ass(F^e_*R/IF^e_*R)$ (even if $R$ is singular). The module $F^e_*R$ being flat implies $\Ass(F^e_*R/IF^e_*R)=\Ass(I)$. Therefore $I^{[p^e]}=I^{\left([p^e]\right)}$. The claim $I_e(I)=I^{[p^e]}$ can then be checked locally by (\ref{splitting ideals and localization}), i.e., we can assume $R$ is a regular local ring. In which case $F^e_*R$ is a free $R$-module. Then $r\not\in I^{[p^e]}$ if and only if $F^e_*r\not \in IF^e_*R$ if and only if there is a choice of projection onto a free summand $\pi:F^e_*R\to R$ so that $\pi(F^e_*r)\not \in I$.

    For statement (\ref{splitting ideals and intersection}), let $\Lambda$ be a non-empty collection of ideals of $R$. Then $r\in I_e\left(\bigcap_{I\in\Lambda}I\right)$ if and only if for all $\varphi\in\Hom_R(F^e_*R,R)$, $\varphi(F^e_*r)\in \bigcap_{I\in\Lambda}I$ if and only if $r\in \bigcap_{I\in\Lambda}I_e(I)$. Statement (\ref{splitting ideals some comparisons 1}) follows from (\ref{splitting ideals containment of bracket powers}) and (\ref{associated primes of splitting ideals}).

    To prove (\ref{splitting ideals some comparisons 2}) we assume $r\not\in I_{e}(I_{e_1}(I))$. There exists $\varphi\in\Hom_R(F^e_*R,R)$ so that $\varphi(F^e_*r)\not \in I_{e_1}(I)$. Therefore there exists $\psi:F^{e_1}_*R\to R$ so that $\psi(F^{e_1}_*\varphi(F^e_*r))\not \in I$. Then $\lambda:=\psi\circ F^{e_1}_*\varphi\in\Hom_R(F^{e+e_1}_*R, R)$ is so that $\lambda(F^{e+e_1}_*r)\not \in I$. Therefore $r\not\in I_{e+e_1}(I)$.

    For statement (\ref{splitting ideals some comparisons 3}), by (\ref{associated primes of splitting ideals}), we can replace $R$ by $R_W$ and may assume that $R$ is $F$-pure. Suppose that $r\not\in I_e(I)$. Let $\varphi:F^e_*R\to R$ be so that $x:=\varphi(F^e_*r)\not\in I$. The property $R$ is $F$-pure implies there exists $R$-linear map $\psi:F^{e_1}_*R\to R$ so that $\psi(F^{e_1}_*1)=1$. Let $\lambda$ denote the following composition of maps,
    \[
    \lambda: F^{e+e_1}_*R\xrightarrow{F^{e_1}_*\varphi} F^{e_1}_*R \xrightarrow{\cdot F^{e_1}_*x^{p^{e_1}-1}}F^{e_1}_*R\xrightarrow{\psi} R.
    \]
    Then $\lambda(F^{e+e_1}_*r)=x\not\in I$, hence $r\not\in I_{e+e_1}(I)$.

    To prove statement (\ref{splitting ideals multiplication containment}), if $c\in R$ and $r\in I_{e}(I)$, then for every $\varphi\in\Hom_R(F^{e}_*R,R)$, $\varphi(F^{e}_*c^{p^{e}}r)=c\varphi(F^e_*r)\in cI$. Therefore $c^{p^e}I_e(I)\subseteq I_e(cI)$ as claimed.

    We give $R[T]$ the standard grading over $R$ in statement (\ref{splitting ideals adjoin variable}). First assume $f(T)\in R[T]\setminus I_e(I)R[T]$. Suppose that $f(T)=\sum f_iT^i$ and $f_c\not\in I_e(I)$. Fix an $R$-linear map $\varphi: F^e_*R\to R$ so that $\varphi(F^e_*f_c)\not\in I$. Construct an $R[T]$-linear map $\varphi_T: F^e_*R[T]\to R[T]$ defined on homogeneous elements, and extending linearly, by
    \[
    \varphi_T(F^e_*g T^\ell) =
    \begin{cases} 
    \varphi(F^e_*g)T^{\left\lfloor \frac{\ell}{p^e} \right\rfloor} & \text{if } \ell \equiv c\mod{p^e} \\
    0 & \text{if } \ell \not\equiv c\mod{p^e}
    \end{cases}.
    \]
    Then $\varphi_T$ an $R[T]$-linear map so that $\varphi_T(F^e_*f(T))\not\in IR[T]$.

    Conversely, assume $f(T)\in R[T]\setminus I_e(IR[T])$ and let $\psi:F^e_*R[T]\to R[T]$ be an $R[T]$-linear map so that $\psi(F^e_*f(T))\not\in IR[T]$. Suppose that $f(T)=\sum f_iT^i$. Then there exists an index $c$ so that $\psi(F^e_*f_cT^c)\not \in IR[T]$. As an $R$-module, $R[T]$ is free with basis $\{T^j\}_{j\in\NN}$. Hence for an appropriate choice of projection $\pi: R[T]\to R$, $\pi(\psi(F^e_*f_cT^c))\not \in I$. Let $\varphi$ be the composition of $R$-linear maps,
    \[
    \varphi: F^e_*R\xrightarrow{F^e_*r\mapsto F^e_*rT^c} F^e_*R[T]\xrightarrow{\psi} R[T]\xrightarrow{\pi} R.
    \]
    Then $\varphi(F^e_*f_c)\not \in I$, implying $f_c\not\in I_e(I)$. Therefore $f(T)\not\in I_e(I)R[T]$.
\end{proof}

The following proposition is independent of the proofs of the main theorems, but illustrates interesting behavior of splitting ideals in Gorenstein rings. Recall that if $R$ is an $F$-finite non-singular domain, then Kunz's Theorem implies for all $e\in\NN$, $I_e(I)=I^{[p^e]}$. Conversely, if $(R,\fm,k)$ is local, then $I_e(\fm)=\fm^{[p^e]}$ is easily checked to be equivalent to $F^e_*R$ being a free $R$-module, equivalently $R$ is regular by Kunz's theorem. When $R$ is singular, then there can still exist large classes of ideals $I$ with the property $I_e(I)=I^{[p^e]}$ for all $e\in\NN$.

\begin{proposition}
    \label{proposition splitting ideals parameters gorenstein}
    Let $R$ be an $F$-finite Gorenstein ring of prime characteristic $p>0$ and $\underline{x}=x_1,x_2,\ldots,x_h$ a regular sequence. Then for all $e\in\NN$, $I_e((\underline{x}))=(\underline{x})^{[p^e]}$.
\end{proposition}

\begin{proof}
    The claimed equality can be checked locally by Lemma~\ref{lemma Some properties of splitting ideals}~(\ref{splitting ideals and localization}). We therefore assume $(R,\fm,k)$ be a local $F$-finite Gorenstein ring. A standard application of \cite[Theorem~3.3.10]{BrunsHerzog} is if $M$ is a finitely generated Cohen-Macaulay $R$-module, then there exists a short exact sequence of Cohen-Macaulay $R$-modules, $0 \to M \xrightarrow{\varphi} R^{\oplus t} \to C \to 0$. A sequence is constructed by surjecting a free module $R^{\oplus t}$ onto $\Hom_R(M, R)$ and then applying the functor $\Hom_R(-, R)$ to the resulting short exact sequence. The module $C$ being Cohen-Macaulay implies $\Tor_1(C, R/(\underline{x})) = 0$. Equivalently, if $\eta \in M \setminus (\underline{x})M$, then $\varphi(m) \in R^{\oplus t} \setminus (\underline{x})R^{\oplus t}$, and projecting onto an appropriate summand of $R^{\oplus t}$ defines an $R$-linear map $\psi: M \to R$ such that $\psi(m) \in R \setminus (\underline{x})$. 
    
    As an $R$-module, $F^e_*R$ is Cohen-Macaulay, and $r \in R \setminus (\underline{x})^{[p^e]}$ if and only if $F^e_*r \in F^e_*R \setminus (\underline{x})F^e_*R$. Therefore if $r\in R\setminus (\underline{x})^{[p^e]}$ then there exists an $R$-linear map $\psi: F^e_*R \to R$ such that $\psi(F^e_*r) \notin (\underline{x})$. Consequently, $I_e(\underline{x}) \subseteq (\underline{x})^{[p^e]}$. By Lemma~\ref{lemma Some properties of splitting ideals} (\ref{splitting ideals containment of bracket powers}), $(\underline{x})^{[p^e]}=I_e((\underline{x}))$.
\end{proof}

\subsection{Connections with Test Ideals}\label{Section connections with test ideals}
Splitting ideals are the foundations of prime characteristic test ideal theory. Let $R$ be an $F$-finite domain of prime characteristic $p > 0$ and Krull dimension $d$. Given an ideal $I \subseteq R$, define
    \[
    I_{[1/p^e]}(I) := \bigcap_{\left\{\fb \subseteq R \ \middle| \ I \subseteq I_e(\fb)\right\}}\fb.
    \]
An implication of Lemma~\ref{lemma Some properties of splitting ideals} (\ref{splitting ideals and intersection}) is that $I\subseteq I_e(I_{[1/p^e]}(I))$. Therefore $I_{[1/p^e]}(I)$ enjoys the following alternative description.
    \begin{proposition}
        \label{prop frobenius roots and splitting ideals}
        Let $R$ be an $F$-finite domain of prime characteristic $p>0$ and Krull dimension $d$. If $I\subseteq R$ is an ideal and $e\in\NN$, then $I_{[1/p^e]}(I)$ is the unique smallest ideal $\fb$ satisfying $I \subseteq I_e(\fb)$.
    \end{proposition}

    If $R$ is non-singular and $\fb$ is any ideal, then $I_e(\fb) = \fb^{[p^e]}$. Hence, in the non-singular case, $I_{[1/p^e]}(I)$ is the unique smallest ideal $\fb$ such that $I \subseteq \fb^{[p^e]}$. It is common to write $I_{[1/p^e]}(I)$ as $I^{[1/p^e]}$ when $R$ is non-singular, see \cite[Definition~2.2]{BlickleMustataSmith}.
    
    Given an ideal $I \subseteq R$ and a real number $t > 0$, the \emph{test ideal} of the pair $(R, I^t)$, denoted $\tau(R, I^t)$, is benefited by several characterizations. Notably, $\tau(R,I^t)$ is realized as the smallest non-zero ideal $J \subseteq R$ such that for all $e \in \NN$ and $\varphi\in\Hom_R(F^e_*R,R)$, 
    \[
    \varphi(F^e_*I^{\lceil t(p^e-1) \rceil}J)\subseteq J,
    \]
    see \cite[Theorem~6.3]{SchwedeCenters}. Equivalently, $\tau(R, I^t)$ is the smallest non-zero ideal of $R$ so that for all $e\in\NN$,
    \[
    I_{[1/p^e]}\left(I^{\lceil t(p^e - 1) \rceil}J\right) \subseteq J.
    \]
    When $R$ is non-singular of dimension $d$, then test ideal theory provides proof that for all ideals $I\subseteq R$ and $n\in\NN$, $I^{(dn)}\subseteq I^n$. If $I\subseteq R$ and $n \in \NN$ then there are ideal containments:
    \[
    I^{(dn)} \subseteq \tau(R, I^{(dn)}) \subseteq \tau(R, (I^{(dn)})^{1/n})^n \subseteq I^n.
    \]
    The first containment holds under milder conditions, such as when $R$ is strongly $F$-regular. The third containment is an application of the Brian\c{c}on-Skoda Theorem. The second containment requires the ring to be regular and follows from the subadditivity theorem, see \cite[Theorem~4.5]{HaraYoshida}. Subadditivity of multiplier/test ideals is not generally true in singular rings, see \cite{TWSubadditivityFails, Hsiao}.

\subsection{Uniform Properties of Noetherian Rings}

In \cite{HunekeUniformBounds}, Huneke brought to light ``deeper and hidden finiteness properties'' of a Noetherian ring $R$ that go exceedingly beyond the definition all ideals of $R$ are finitely generated. 

\begin{definition}[Uniform Properties]
    \label{definition uniform properties Artin Rees and BS and USTP}
    Let $R$ be a Noetherian ring of finite Krull dimension.
    \begin{itemize}
        \item An ideal $J\subseteq R$ has the \emph{Uniform Artin-Rees Property} if there exists a constant $A$ so that for every ideal $I\subseteq R$ and $n\in\mathbb{N}$,
        \[
        J\cap I^{n+A}\subseteq JI^n.
        \]
        The constant $A$ is a \emph{Uniform Artin-Rees Bound} of the ideal $J\subseteq R$. The ring $R$ has the \emph{Uniform Artin-Rees Property} if every ideal of $R$ enjoys the Uniform Artin-Rees Property.

        \smallskip
        
        \item The ring $R$ has the \emph{Uniform Brian\c{c}on-Skoda Property} if there is a constant $B$ so that for all ideals $I\subseteq R$ for all $n\in\mathbb{N}$,
        \[
        \overline{I^{n+B}}\subseteq I^n.
        \]
        The constant $B$ is a \emph{Uniform Brian\c{c}on-Skoda Bound} of $R$.
    \end{itemize}
\end{definition}

Huneke's introduction of the Uniform Artin-Rees and Brian\c{c}on-Skoda properties were solidified as fundamental characteristics of several large classes of rings, see \cite[Theorem~4.12 and Theorem~4.13]{HunekeUniformBounds}. In particular, the Uniform Artin-Rees Property and the Uniform Brian\c{c}on-Skoda Property are applicable in our investigations of $F$-finite rings.

\begin{theorem}[Huneke's Uniform Theorems and $F$-finite Rings]
    \label{Uniform Theorems}
    Let $R$ be an $F$-finite ring of prime characteristic $p>0$.
    \begin{enumerate}
        \item \cite[Theorem~4.12]{HunekeUniformBounds}, $R$ enjoys the Uniform Artin-Rees Property.
        \item \cite[Theorem~4.13]{HunekeUniformBounds}, If $R$ is reduced then $R$ enjoys the Uniform Brian\c{c}on-Skoda Property.
    \end{enumerate}
\end{theorem}

\subsection{Symbolic Multipliers}

Proofs of the main theorems benefit from the existence of elements that uniformly multiply symbolic powers of ideals into a linearly bounded difference of ordinary powers of ideals. Such elements are referred to as \emph{uniform symbolic multipliers} and the existence of such elements are central to arguments found in \cite{HKV, HKAbelian, HKHypersurface, PolstraZariskiNagata}.

\begin{definition}
    \label{definition uniform multiplier}
    Let $R$ be a Noetherian ring and $z\in R$ an element that avoids all minimal primes of $R$. 
    \begin{itemize}
        \item We say that $z$ is a \emph{uniform symbolic multiplier} if there exists a constant $C$ so that for all ideals $I\subseteq R$ and $n\in \NN$,
        \[
        z^nI^{(Cn)}\subseteq I^n.
        \]
        The constant $C$ is called an \emph{uniform symbolic multiplier constant of $z$}.
        \item If $R$ has prime characteristic $p>0$, then $z$ is a \emph{uniform Frobenius symbolic multiplier} if for all $e\in\NN$ and $I\subseteq R$,
        \[
        z^{p^e}I^{\left([p^e]\right)}\subseteq I^{[p^e]}.
        \]
    \end{itemize}

\end{definition}

\begin{remark}
    The definition of a uniform symbolic multiplier given in Definition~\ref{definition uniform multiplier} is a weaker form of the definition provided by \cite[Definition~3.1]{HKAbelian}. 
\end{remark}

As indicated in the proof of {\cite[Lemma~3.3]{HKAbelian}, minor modifications to the proofs of \cite[Proposition~3.4 and Theorem~3.5]{HKV} prove the existence of uniform symbolic and uniform Frobenius symbolic multipliers. The following lemma contains a modest improvement of {\cite[Lemma~3.3]{HKAbelian} that is necessary to the proof of Main Theorem~\ref{Main Theorem USTP for splitting ideals}. 

\begin{lemma}[{\cite[Lemma~3.3]{HKAbelian}}]
    \label{lemma uniform Frobenius multipliers}
    Let $R$ be an $F$-finite domain and $\JJ\subseteq R$ the reduced ideal defining the singular locus of $\Spec(R)$. There exists $e_0\in\NN$ so that the nonzero elements of $\JJ^{[p^{e_0}]}$ are uniform symbolic multipliers and uniform Frobenius symbolic multipliers of $R$. Moreover, $e_0\in\NN$ can be chosen so that for all ideals $I\subseteq R$ and $e\in\NN$,
    \[
    \JJ^{[p^{e+e_0}]}I_e(I)\subseteq I^{[p^e]}.
    \]
\end{lemma}

\begin{proof}
    The claim there exists $e_0\in\NN$ so that the nonzero elements of $\JJ^{[p^{e_0}]}$ are uniform symbolic multipliers and uniform Frobenius symbolic multipliers is the content of \cite[Lemma~3.3]{HKAbelian}. For the remaining claim, it suffices to show that if $c$ is a minimal generator of $\JJ$ then there exists $e_0\in\NN$ so that for all ideals $I\subseteq R$ and $e\in\NN$, $c^{p^{e+e_0}}I_e(I)\subseteq I^{[p^e]}$. If $0\not=c\in\JJ$ then $F_*R_c$ is a flat $R_c$-module, i.e., locally free. The spectrum $\Spec(R)$ is then covered by finitely many affine open sets $\Spec(R_i)$ so that if $R_{i,c}$ is the localization of $R_i$ with respect to the element $c$, then $F_*R_{i,c}$ a free $R_{i,c}$-module. By standard arguments, see \cite[Proof of Proposition~3.4 top of page 334]{HKV}, there exists a constant $t_i\in\NN$ so that for every $e\in\NN$ there exists a free $R_{i}$-module $F_{e,i}\subseteq F^e_*R_{i}$ and $c^{t_i}F^e_*R_{i,c}\subseteq F_{e,i}$. Choose a constant $e_0$ so that $p^{e_0}\geq t_i$ for all $i$. We claim that for every ideal $I\subseteq R$ and $e\in\NN$, $c^{p^{e+e_0}}I_e(I)\subseteq I^{[p^e]}$. 

    For a fixed ideal $I\subseteq R$ and $e\in\NN$, the containment $c^{p^{e+e_0}}I_e(I)\subseteq I^{[p^e]}$ can be checked after localization at an associated prime $\fp$ of $I^{[p^e]}$. The affine open sets $\Spec(R_i)$ cover $\Spec(R)$, therefore $\fp\in \Spec(R_i)$ for some $i$. If $x\in R_\fp\setminus I^{[p^e]}R_\fp$ then $F^e_*x\in F^e_*R_\fp\setminus IF^e_*R_{\fp}$. If $F_{e,i,\fp}$ is the localization of $F_{e,i}$ with respect to $\fp$, then $F^e_*x\not\in IF_{e,i,\fp}$. Therefore $c^{p^{e_0}}F^e_*x\in F_{e,i,\fp}\setminus c^{p^{e_0}}IF_{e,i,\fp}$. Therefore there exists a projection map $\pi:F_{e,i,\fp}\to R_\fp$ so that $\pi(c^{p^{e_0}}F^e_*x)\not\in c^{p^{e_0}}I R_\fp$. Let $\varphi$ be the composition of $R_\fp$-linear maps
    \[
    \varphi: F^e_*R_\fp \xrightarrow{\cdot c^{p^{e_0}}} F_{e,i,\fp}\xrightarrow{\pi} R_\fp,
    \]
    then $\varphi(F^e_*r) = \pi(c^{p^{e_0}}F^e_*x)\not\in c^{p^{e_0}}I R_\fp$. Therefore there is an inclusion of ideals
    \[
    I_e(c^{p^{e_0}}IR_\fp)\subseteq I^{[p^e]}R_\fp.
    \]
    By Lemma~\ref{lemma Some properties of splitting ideals} (\ref{splitting ideals multiplication containment}),
    \[
    c^{p^{e+e_0}}I_e(IR_\fp)\subseteq I_e(c^{p^{e_0}}IR_\fp)\subseteq I^{[p^e]}R_\fp.
    \]
    By Lemma~\ref{lemma Some properties of splitting ideals}~(\ref{splitting ideals and localization}),
    \[
    c^{p^{e+e_0}}I_e(I)R_\fp=c^{p^{e+e_0}}I_e(IR_\fp)\subseteq I^{[p^e]}R_\fp.
    \]
\end{proof}

It remains an open problem whether a sufficiently large class of singular domains not containing a field admits a uniform symbolic multiplier. Significant advancements in the Uniform Symbolic Topology Property for equicharacteristic singular rings often rely on the existence of uniform symbolic multipliers, the Uniform Artin Theorem, and the Uniform Brian\c{c}on-Skoda Theorem, cf. \cite{HKV, HKAbelian, HKHypersurface, PolstraZariskiNagata}. Therefore, progressing the Uniform Symbolic Topology Property for singular rings of mixed characteristic may be best suited to rings that are essentially of finite type over an excellent local ring or over $\mathbb{Z}$, as such rings are known to satisfy the Uniform Artin-Rees and Uniform Brian\c{c}on-Skoda Properties by \cite[Theorems~4.12 and 4.13]{HunekeUniformBounds}.

\begin{question}\label{question symbolic multipliers exist}
    Let $R$ be a domain essentially of finite type over an excellent local ring or $\ZZ$. 
    \begin{enumerate}
        \item\label{usm exist} Does $R$ admit a uniform symbolic multiplier?
        \item\label{usm jacobian ideal} If $0\not =c\in R$ and $R_c$ is non-singular, is a power of $c$ a uniform symbolic multiplier?
    \end{enumerate} 
\end{question}

If Question~\ref{question symbolic multipliers exist}~(\ref{usm jacobian ideal}) has a positive answer, then the Uniform Symbolic Topology Property of non-singular rings, quasi-compactness of $\Spec(R)$, and the methods of \cite{HKV} are sufficient to assert that if $R$ is reduced, essentially of finite type over $\ZZ$, 
 and $R_\fp$ is non-singular for all non-maximal prime ideals $\fp\in \Spec(R)$, then $R$ enjoys the Uniform Symbolic Topology Property.

\section{On the Powers of Ideals, Frobenius Powers of Ideals, and Splitting Ideals}\label{Section Groundwork on powers and splitting ideals}

Theorem~\ref{Theorem Splitting ideals and powers of ideals} applies to all $F$-finite domains and primarily follows from Lemma~\ref{lemma uniform Frobenius multipliers}, the Uniform Artin-Rees theorem, and Uniform Brian\c{c}on-Skoda theorem. The theorem contributes to our progress on the Uniform Symbolic Topology Problem in a manner indicated by Huneke and Katz's Bootstrapping Theorem, \cite[Theorem~3.5]{HKAbelian}.

\begin{theorem}
    \label{Theorem Splitting ideals and powers of ideals}
    Let $R$ be an $F$-finite domain of prime characteristic $p>0$. There exists a constant $C$ so that for all ideals $I\subseteq R$ and $e\in\NN$,
    \[
    I_e(I^C)\subseteq I^{[p^e]}.
    \]
\end{theorem}

\begin{proof}
    By Lemma~\ref{lemma Some properties of splitting ideals} (\ref{splitting ideals adjoin variable}) we can adjoin a variable $T$ to $R$ and only consider ideals $I\subseteq R$ so that for all $e\in\NN$ and for all $\fp\in \Ass(I^{[p^e]})$, $R_\fp$ has infinite residue field, cf. \cite[Proof of Lemma~2.4 (b)]{HHComparison}. Suppose that $R$ has Krull dimension $d$. Let $0\not=c\in R$ be an element so that for all ideals $I\subseteq R$ and $e\in\NN$, $c^{p^e}I_e(I)\subseteq I^{[p^e]}$, see Lemma~\ref{lemma uniform Frobenius multipliers}. Let $A$ be a Uniform Artin-Rees bound of $(c)\subseteq R$. For all ideals $I\subseteq R$ and $n\in\NN$, $(c)\cap I^{n+A} = c(I^{n+A}:_Rc)\subseteq cI^n$. Canceling $c$ and an inductive argument implies that for all ideals $I\subseteq R$ and $n,t\in \NN$,
    \begin{align}\label{equation consequence of uniform artin rees}
        (I^{n+At}:_Rc^t)\subseteq I^n.
    \end{align}

    Let $B$ be a Uniform Brian\c{c}on-Skoda Bound of $R$. Then for all ideals $I\subseteq R$ and $e\in\NN$, $c^{p^e}I_e(I^{d+A+B})\subseteq (I^{d+A+B})^{[p^e]}$. Therefore
    \begin{align*}
        I_e(I^{d+A+B}) &\subseteq ((I^{d+A+B})^{[p^e]}:_Rc^{p^e})\\
        &\subseteq ((I^{d+A+B})^{p^e}:_Rc^{p^e})\\
        &\subseteq I^{(d+B)p^e} \mbox{ {\footnotesize by (\ref{equation consequence of uniform artin rees}).}}
    \end{align*}
    
    For all $\fp\in\Ass(I^{[p^e]})$, the residue field of $R_\fp$ is infinite. Therefore $I$ admits a minimal reduction by at most $d$ elements at the localization of an associated prime of $I^{[p^e]}$. Integral closure of ideals and powers of ideals commute with localization. Therefore the constant $B$ is a Uniform Brian\c{c}on-Skoda Bound of all localizations of $R$. Therefore
    \[
    I_e(I^{d+A+B})\subseteq I^{(d+B)p^e}\subseteq \overline{I^{dp^e+B}}\subseteq I^{[p^e]}.
    \]
    The constant $C=d+A+B$ was independent of the choice of ideal $I\subseteq R$.
\end{proof}

Let $(R,\fm,k)$ be an $F$-finite domain of prime characteristic $p>0$ and Krull dimension $d>0$. Let $I\subseteq R$ be an $\fm$-primary ideal. In \cite[Section~2]{PolstraSmirnovEquimultiplicity} the authors define the \emph{$F$-signature of $I$} as $s(I)=\lim_{e\to\infty}\frac{\ell(R/I_e(I))}{p^{ed}}$. The authors prove the following about the $F$-signature of an $\fm$-primary ideal:
\begin{itemize}
    \item The limit $s(I)$ exists;
    \item $\ell(R/I_e(I)) = s(I)p^{ed}+O(p^{e(d-1)})$;
    \item If $s(I) = \lim_{e\to\infty}\frac{e_{HK}(I_e(I))}{p^{ed}}$ where $e_{HK}(I_e(I))$ is the Hilbert-Kunz multiplicity of the $\fm$-primary ideal $I\subseteq R$.
\end{itemize}

Theorem~\ref{Theorem Splitting ideals and powers of ideals} then implies positivity of the $F$-signature of all $\fm$-primary ideals contained in a deep enough power of the maximal ideal.

\begin{corollary}
    \label{cor: positivity of F-signature of deep enough ideals}
    Let $(R,\fm,k)$ be an $F$-finite local domain. There exists $C\in\NN$ so that for all $\fm$-primary ideals $I\subseteq \fm^C$, the $F$-signature of $I$ is positive.
\end{corollary}

\begin{proof}
    By Theorem~\ref{Theorem Splitting ideals and powers of ideals}, only applied to the maximal ideal $\fm$, there exists $C\in\NN$ so that $I_e(\fm^C)\subseteq \fm^{[p^e]}$. If $I\subseteq \fm^{C}$ then $I_e(I)\subseteq I_e(\fm^C)\subseteq \fm^{[p^e]}$. If $d$ is the Krull dimension of $R$, and $I$ is $\fm$-primary, then
    \[
    s(I)=\lim_{e\to \infty}\frac{\ell(R/I_e(I))}{p^{ed}}\geq \lim_{e\to \infty} \frac{\ell(R/\fm^{{[p^e]}})}{p^{ed}} = e_{HK}(R)\geq 1.
    \]
\end{proof}

If $R$ is a Noetherian ring of prime characteristic $p>0$ and $I\subseteq R$ is an ideal, then $I^{[p^e]}\subseteq I^{p^e}$ and $I^{\left([p^e]\right)}\subseteq I^{(p^e)}$. Standard arguments involving reductions, cf. \cite[Proof of Lemma~2.4 (b)]{HHComparison}, and the Uniform Brian\c{c}on-Skoda Property, we can observe a uniform linear adjustment of the symbolic powers of an ideal are contained in the Frobenius symbolic power of an ideal.

\begin{lemma}
    \label{lemma uniform symbolic Frobenius power multipliers}
    Let $R$ be an $F$-finite domain and of prime characteristic $p>0$. There exists a constant $C$, depending only on the dimension of $R$ and a Uniform Brian\c{c}on-Skoda bound of $R$, so that for all ideals $I\subseteq R$ and $e\in\NN$, $I^{(Cp^e)}\subseteq I^{\left([p^e]\right)}$.
\end{lemma}

\begin{proof}
    By passing to a polynomial extension $R[T]$, we can assume that for each $\fp\in\Ass(I)$, the residue field of $R_\fp$ is infinite. Let $B$ be a Uniform Brian\c{c}on-Skoda Bound of $R$. After localization at an associated prime of $I$, and recalling that an $F$-finite domain has finite Krull dimension, we may assume $(R,\fm,k)$ is local of Krull dimension $d$ with infinite residue field and show that for all ideals $I\subseteq R$ and $e\in\NN$, $I^{(d+B)p^e}\subseteq I^{[p^e]}$. Every ideal has a reduction by at most $d$ elements, hence there is an ideal $J\subseteq I$ that is at most $d$-generated so that $I^{(d+B)p^e}\subseteq I^{dp^e+B}\subseteq \overline{J^{dp^e+B}}\subseteq J^{dp^e}\subseteq J^{[p^e]}\subseteq I^{[p^e]}$.
\end{proof}

Our next lemma is analogous to \cite[Main Theorem]{SwansonLinear} and specific to splitting ideals and Frobenius symbolic powers of ideals in rings of prime characteristic.

\begin{lemma}
    \label{lemma standard linear comparision of splitting ideals and Frobenius Powers}
    Let $R$ be an $F$-finite domain of prime characteristic $p>0$. Let $I\subseteq R$ an ideal, $W$ the complement of the union of the associated primes of $I$, and assume that the adic and symbolic topologies of $I$ are equivalent.
    \begin{enumerate}
        \item\label{standard lin comparison non F reg} There exists a constant $e_0\in\NN$ so that for all $e\in\NN$, $I^{\left([p^{e+e_0}]\right)}\subseteq I^{[p^e]}$.
        \item\label{standard lin comparison F reg} If $R_W$ is strongly $F$-regular, then there exists an $e_1\in\NN$ so that for all $e\in\NN$, $I_{e+e_1}(I)\subseteq I^{[p^e]}$.
    \end{enumerate} 
\end{lemma}

\begin{proof}
    For all $e,e_0\in \NN$, there is a containment of the Frobenius symbolic power and symbolic powers $I^{\left([p^{e+e_0}]\right)} \subseteq I^{(p^{e+e_0})}$. By \cite[Main Result]{SwansonLinear} there exists $C\in\NN$ so that for all $n\in\NN$, $I^{(Cn)}\subseteq I^{n}$. Choose a constant $D$ so that $I^{Dp^e}\subseteq I^{[p^e]}$ for all $e\in\NN$ and let $e_0\in \NN$ be so that $p^{e_0}\geq CD$. Then for all $e\in\NN$,
    \[
    I^{\left([p^{e+e_0}]\right)} \subseteq I^{(CDp^e)}\subseteq I^{Dp^e}\subseteq I^{[p^e]}.
    \]

    Now suppose that $R_W$ is strongly $F$-regular. By the above, it suffices to find a constant $e_2$ so that $I_{e+e_2}(I)\subseteq I^{\left([p^e]\right)}$ for all $e\in\NN$. The associated primes of $I_{e}(I)$ are the associated primes of $I$ by Lemma~\ref{lemma Some properties of splitting ideals} (\ref{associated primes of splitting ideals}). We therefore can determine the desired containment property in the localization $R_W$. Let $C$ be as in Theorem~\ref{Theorem Splitting ideals and powers of ideals} with respect to the localized ring $R_W$. The localized ring $R_W$ is strongly $F$-regular, hence $\bigcap_{e\in\NN} I_{e}(I)R_W= 0$. By Chevalley's Lemma, \cite[Lemma~7]{Chevalley}, there exists an $e_2\in\NN$ so that $I_{e_2}(I)R_W\subseteq I^CR_W$. By Lemma~\ref{lemma Some properties of splitting ideals} (\ref{splitting ideal containment}) and (\ref{splitting ideals some comparisons 2}), for every $e\in\NN$ there is a containment of ideals
    \[
    I_{e+e_2}(I)R_W\subseteq I_e(I_{e_2}(I))R_W\subseteq I_e(I^C)R_W\subseteq I^{[p^e]}R_W.
    \]
\end{proof}


If $\fq$ is a prime ideal of an $F$-finite domain $R$ belonging to the non-singular locus of $\Spec(R)$, then there exists an element $c\in R\setminus \fq$ so that for all ideals $I\subseteq R$ and $e\in\NN$, $c^{p^e}I_e(I)\subseteq I^{[p^e]}$, see Lemma~\ref{lemma uniform Frobenius multipliers}. The next lemma will serve as a substitute to this phenomena for a prime ideal $\fq$ belonging to the strongly $F$-regular locus of $\Spec(R)$. The below lemma is the fundamental blockade from extending the main theorems of this article from isolated non-strongly $F$-regular singularities to all isolated non-normal $F$-finite domains. Theorem~\ref{Theorem Splitting ideals and powers of ideals} and Lemma~\ref{lemma standard linear comparision of splitting ideals and Frobenius Powers} (\ref{standard lin comparison F reg}) provide suitable conditions for the assumptions of the lemma to be enjoyed. 

\begin{lemma}
    \label{lemma multiplier for singular locus prime}
    Let $R$ be an $F$-finite domain and $\fq\in\Spec(R)$. Assume that $e_0\in\NN$ is so that for all ideals $\fa\subseteq R$ and $e\in\NN$, $I_e(\fa^{[p^{e_0}]})\subseteq \fa^{[p^e]}$. Assume $e_1\geq e_0$ has the property that $I_{e_1}(\fq)\subseteq \fq^{[p^{e_0}]}$. There exists element $c\in R\setminus \fq$ so that for all ideals $I\subseteq R$ and $e\in\NN$,
    \[
    c^{p^e}I_e(I_{e_1}(I))\subseteq (I,\fq)^{[p^e]}.
    \]
\end{lemma}

\begin{proof}
    Let $F'$ be a free $R_\fq$-module of largest possible rank appearing as a direct summand of $F^{e_1}_*R_\fq$. Suppose $F^{e_1}_*R_\fq\cong F'\oplus N'$. Then $N'$ does not omit a free $R_\fq$-summand, i.e., $\Hom_{R_\fq}(N', R_\fq) = \Hom_{R_\fq}(N', \fq R_\fq)$. Choose $F$ to be the free submodule of $F^{e_1}_*R$ by clearing denominators of a basis of $F'$. If $N$ is a choice of lift of $N'$, let 
    \[
    \widetilde{N}=  \{F^{e_1}_*r\in F^{e_1}_*R\mid \exists \, c\in R\setminus \fq, \, cF^{e_1}_*r\in N\}.
    \]
    Then $\widetilde{N}$ is a submodule of $F^{e_1}_*R$ so that $N\subseteq \widetilde{N}$ and $N'=N_\fq =\widetilde{N}_\fq$. Replace $N$ by $\widetilde{N}$ so that our choice of lift of $N'$ has the property that for all $c\in R\setminus \fq$,
    \[
    (N:_{F^{e_0}_*R}c)=N.
    \]
    As $N_\fq=N'$ does not omit a free $R_\fq$-summand, $\Hom_R(N,R)= \Hom_R(N,\fq)$. In particular, if $\varphi\in\Hom_R(F^{e_1}_*R,R)$ then $\varphi(N)\subseteq \fq$. Equivalently, $N\subseteq F^{e_1}_*I_{e_1}(\fq)$. 
    
    The submodule $F+N\subseteq F^{e_1}_*R$ is necessarily an internal direct summand, else a non-trivial combination of elements of $F$ and $N$ in $F^{e_1}_*R$ localizes to a non-trivial combination of elements of $F'$ and $N'$. This provides a containment of $R$-modules $F\oplus N\subseteq F^{e_1}_*R$ that agree upon localization at $\fq$. Therefore there exists $c\in R\setminus \fq$ so that $cF^{e_1}_*R\subseteq F\oplus N$.

    Suppose that $r\not \in (I^{[p^{e_1}]}, I_{e_1}(\fq))$. Then $F^{e_1}_*r\not \in IF^{e_1}_*R + F^{e_1}_*I_{e_1}(\fq)\supseteq (I,\fq)F \oplus N$. Hence $cF^{e_1}_*r\in F\oplus N$ avoids $c((I,\fq)F \oplus N)$. Assume that $cF^{e_1}_*r=f+n$ with $f\in F$ and $n\in N$. We claim that $f\not\in c(I,\fq)F$. Else there exists $\lambda\in (I,\fq)F$ so that $cF^{e_1}_*r-c\lambda \in N$. Hence $F^{e_1}_*r-\lambda\in (N:_{F^{e_1}_*R}c) = N$. Contradicting that $F^{e_1}_*r\not \in (I,\fq)F\oplus N$. Therefore there exists a choice of projection map $\pi:F\oplus N\to F\to R$ so that $\pi(cF^{e_1}_*r)\not \in c(I,\fq)$. The composition of $R$-linear maps
    \[
    \varphi:F^{e_1}_*R\xrightarrow{\cdot c} F\oplus N\xrightarrow{\pi} R
    \]
    is so that $\varphi(F^{e_1}_*r)\not \in c(I,\fq)$. Therefore $r\not\in I_{e_1}(c(I,\fq))$. Consequently, there are containments of ideals
    \begin{equation}
        \label{equation for lemma multiplier singular locus primes}
        I_{e_1}(cI)\subseteq I_{e_1}(c(I,\fq))\subseteq (I^{[p^{e_1}]},I_{e_1}(\fq))\subseteq (I^{[p^{e_1}]},\fq^{[p^{e_0}]})\subseteq (I,\fq)^{[p^{e_0}]}.
    \end{equation}
    For every $e\in\NN$,
    \begin{align*}
    c^{p^{e+e_1}}I_{e}(I_{e_1}(I))&\subseteq I_{e}(c^{p^{e_1}}I_{e_1}(I)) \mbox{ {\footnotesize by Lemma~\ref{lemma Some properties of splitting ideals} (\ref{splitting ideals multiplication containment})}}\\
    &\subseteq  I_e(I_{e_1}(cI)) \mbox{ {\footnotesize by Lemma~\ref{lemma Some properties of splitting ideals} (\ref{splitting ideal containment}) and (\ref{splitting ideals multiplication containment})}}\\
    & \subseteq I_e((I,\fq)^{[p^{e_0}]}) \mbox{ {\footnotesize by (\ref{equation for lemma multiplier singular locus primes})}}\\
    & \subseteq (I,\fq)^{[p^e]} \mbox{ {\footnotesize by assumed property of $e_0$.}}
    \end{align*}
\end{proof}
Let $R$ be an $F$-finite domain. Assume that $e_0 \in \mathbb{N}$ is such that for all ideals $\fa \subseteq R$ and all $e \in \mathbb{N}$, $I_{e_0}(\fa^{[p^{e_0}]}) \subseteq \fa^{[p^e]}.$ If $\fq \in \Spec(R)$ is such that $R_{\fq}$ is not strongly $F$-regular, then it is not difficult to show that there does not exist $e_1 \in \mathbb{N}$ such that $I_{e_1}(\fq) \subseteq \fq^{[p^{e_0}]}.$ In particular, one cannot apply the conclusion of Lemma~\ref{lemma multiplier for singular locus prime} to the prime $\fq$, as the hypotheses of the lemma are not satisfied.

The following question asks whether a weaker property than that stated in Lemma~\ref{lemma multiplier for singular locus prime} might still hold. A positive answer to this question, together with minor adjustments to the proof of Theorem~\ref{theorem USTP in isolated non SFR rings}, would suffice to extend Main Theorem~\ref{Main theorem USTP in SFR rings} to all normal domains of prime characteristic that are either $F$-finite or essentially of finite type over an excellent local ring.

\begin{question}
    \label{question basically USTP for all prime char rings}
    Let $R$ be an $F$-finite normal domain of prime characteristic $p>0$ and $\fq\in\Spec(R)$. Does there exist $c\in R\setminus \fq$ and constants $C,e_1\in\NN$ so that for all ideals $I\subseteq R$ and $e\in\NN$, $c^{p^e}I_{e}(I_{e_1}(I^{(C)}))\subseteq (I,\fq)^{[p^e]}$?
\end{question}

The next lemma is analogous to the Uniform Chevalley Theorem, \cite[Theorem~2.3]{HKV}, as it pertains to splitting ideals of an $F$-finite domain.

\begin{lemma}
    \label{Uniform Chevalley Lemma for Frobenius powers}
    Let $R$ be an $F$-finite domain of prime characteristic $p>0$ and $J\subseteq R$ an ideal primary to a prime ideal $\fp\in\Spec(R)$. Let $\Lambda$ be the collection of ideals $I\subseteq R$ whose symbolic topology is finer than the $J$-adic topology.
    \begin{enumerate}
        \item\label{Uniform Chevalley non-F-reg} There exists constants $D\in\NN$ so that for all ideals $I\in\Lambda$ and $e\in\NN$,
        \[
        I_e(I^{(D)})\subseteq J^{[p^e]}.
        \]
        \item\label{Uniform Chevalley F-reg} If the $J$-adic topology is equivalent to the $J$-symbolic topology of $R$ and if $R_\fp$ is strongly $F$-regular, then there exists a constant $e_1\in\NN$ so that for all ideals $I\in\Lambda$ and $e\in\NN$,
        \[
        I_{e+e_1}(I)\subseteq J^{[p^e]}.
        \]
    \end{enumerate}
\end{lemma}

\begin{proof}
    By Theorem~\ref{Theorem Splitting ideals and powers of ideals} there exists a constant $C$ so that for all $e\in\NN$,
    \begin{equation}\label{containment equation for J and C}
        I_e(J^C)\subseteq J^{[p^e]}.
    \end{equation}
    By \cite[Corollary~2.4]{HKV} there exists a constant $D$ so that for all ideals $I\in \Lambda$, $I^{(D)}\subseteq J^C$. For every $I\in\Lambda$ and $e\in\NN$,
    \begin{align*}
        I_e(I^{(D)})&\subseteq I_e(J^C) \mbox{ {\footnotesize by Lemma~\ref{lemma Some properties of splitting ideals} (\ref{splitting ideal containment})}}\\
        &\subseteq J^{[p^e]} \mbox{ {\footnotesize by (\ref{containment equation for J and C})}}.
    \end{align*}
    
    Now suppose the $J$-adic topology is equivalent to the $J$-symbolic topology of $R$ and $R_\fp$ is strongly $F$-regular. Continue to let $C$ be as in (\ref{containment equation for J and C}) and $D\in\NN$ be so that $J^{(D)}\subseteq J^C$.  The assumption $R_\fp$ is strongly $F$-regular is equivalent to the assumption $\bigcap_{e\in\NN}I_{e}(\fp R_\fp)=0$. By Chevalley's Lemma, \cite[Lemma~7]{Chevalley}, there exists a constant $e_1$ so that $I_{e_1}(\fp R_\fp)\subseteq J^DR_\fp$. If $I\in\Lambda$ then $I\subseteq \fp$. Therefore for all ideals $I\in\Lambda$ and $e\in\NN$,
    \begin{align*}
        I_{e+e_1}(I)R_\fp &= I_{e+e_1}(IR_\fp) \mbox{ {\footnotesize by Lemma~\ref{lemma Some properties of splitting ideals} (\ref{splitting ideals and localization})}}\\
        &\subseteq I_e(I_{e_1}(IR_\fp)) \mbox{ {\footnotesize by Lemma~\ref{lemma Some properties of splitting ideals} (\ref{splitting ideals some comparisons 2})}}\\
        &\subseteq I_e(I_{e_1}(\fp)) \mbox{ {\footnotesize by Lemma~\ref{lemma Some properties of splitting ideals} (\ref{splitting ideal containment})}} \\
        &\subseteq I_e(J^DR_\fp) \\
        &= I_e(J^D)R_\fp \mbox{ {\footnotesize by Lemma~\ref{lemma Some properties of splitting ideals} (\ref{splitting ideals and localization}).}}
    \end{align*}
    By Lemma~\ref{lemma Some properties of splitting ideals} (\ref{associated primes of splitting ideals}), 
    \begin{equation}\label{contracting splitting ideal back to R}
        I_e(J^D)R_\fp\cap R = I_e(J^{(D)}).
    \end{equation}
    Therefore for all $I\in\Lambda$ and $e\in\NN$,
    \begin{align*}
        I_{e+e_1}(I)&\subseteq I_{e+e_1}(I)R_\fp\cap R\\
        &\subseteq I_e(J^D)R_\fp\cap R\\
        &=I_e(J^{(D)}) \mbox{ {\footnotesize by (\ref{contracting splitting ideal back to R})}}\\
        &\subseteq I_e(J^{C}) \mbox{ {\footnotesize by Lemma~\ref{lemma Some properties of splitting ideals} (\ref{splitting ideal containment})}}\\
        &\subseteq J^{[p^e]}.
    \end{align*}
\end{proof}

\subsection{Techniques in proofs of main theorems} The materials covered so far lay the groundwork for a series of lengthy inductive arguments fundamental to the proofs of Main Theorems~\ref{Main theorem USTP in SFR rings} and \ref{Main Theorem USTP for splitting ideals}, see Theorem~\ref{theorem USTP in isolated non SFR rings}. We conclude this section with a proof of these theorems under simplified hypotheses. The purpose of doing so is to clearly demonstrate techniques of applications of the theory developed.

\begin{proposition}[Main Theorems, Special Case]
    \label{proposition ustp sfr restricted assumptions}
    Let $(R,\fm,k)$ be a local $F$-finite strongly $F$-regular ring of prime characteristic $p>0$. Let $\Sing(R)$ be the locus of singular primes of $\Spec(R)$. Assume that $\Sing(R)=V(\fq)$, $\fq\in\Spec(R)$, and $\dim(R/\fq)=1$.
    \begin{enumerate}
        \item\label{Main thm 2 special} There exists $e'\in\NN$ so that for all ideals $I\subseteq R$ and $e\in\NN$, 
        \[
        I_{e+e'}(I)\subseteq I^{[p^e]}.
        \]
        \item\label{Main thm 1 special} The ring $R$ enjoys the Uniform Symbolic Topology Property.
    \end{enumerate}

\end{proposition}

\begin{proof}
    We first prove (\ref{Main thm 2 special}). By Lemma~\ref{lemma uniform Frobenius multipliers}, there exists $e_0\in\NN$ so that for all ideals $I\subseteq R$ and $e\in\NN$,
    \begin{equation}\label{equation multiplier simple}
    \fq^{[p^{e+e_0}]}I_e(I)\subseteq I^{[p^e]}.
    \end{equation}
    By Lemma~\ref{lemma standard linear comparision of splitting ideals and Frobenius Powers}~(\ref{standard lin comparison F reg}) there exists a constant $e_1\in\NN$ so that $I_{e_1}(\fq)\subseteq \fq^{[p^{e_0}]}$. By Lemma~\ref{lemma multiplier for singular locus prime} there exists an element $c\in R\setminus \fq$ so that for all ideals $I\subseteq R$ and $e\in\NN$,
    \begin{equation}\label{singular multiplier simple}
        c^{p^{e}}I_e(I_{e_1}(I))\subseteq (I,\fq)^{[p^e]}.
    \end{equation}
    The assumption $\dim(R/\fq)=1$ implies the ideal $(\fq,c)$ is either $\fm$-primary or the unit ideal. Either scenario, Lemma~\ref{Uniform Chevalley Lemma for Frobenius powers}~(\ref{Uniform Chevalley F-reg}) implies there exists a constant $e_2\in\NN$ so that for all ideals $I\subseteq R$ and $e\in\NN$,
    \begin{equation}\label{uniform chevalley simple}
        I_{e+e_2}(I)\subseteq (\fq,c)^{[p^e]}.
    \end{equation}
    Let $e_3=\max\{e_1,e_2\}$. Then for all $e\in\NN$,
    \begin{align*}
        I_{e+e_0+e_3}(I)^2 &= I_{e+e_0+e_3}(I)I_{e+e_0+e_3}(I) \\
        &\subseteq I_{e+e_0+e_2}(I)I_{e+e_0+e_1}(I) \mbox{ {\footnotesize by Lemma~\ref{lemma Some properties of splitting ideals} (\ref{splitting ideals some comparisons 3})}} \\
        &\subseteq (\fq,c)^{[p^{e+e_0}]}I_{e+e_0+e_1}(I) \mbox{ {\footnotesize by (\ref{uniform chevalley simple})}}\\
        &\subseteq (\fq,c)^{[p^{e+e_0}]}I_{e+e_0}(I_{e_1}(I)) \mbox{ {\footnotesize by Lemma~\ref{lemma Some properties of splitting ideals} (\ref{splitting ideals some comparisons 2})}}\\
        &\subseteq \fq^{[p^{e+e_0}]} + c^{p^{e+e_0}}I_{e+e_0}(I_{e_1}(I))\\
        &\subseteq (I,\fq)^{[p^{e+e_0}]} \mbox{ {\footnotesize by (\ref{singular multiplier simple})}}.
    \end{align*}
    Multiplying the above containment by $I_{e+e_0+e_3}(I)$,
    \begin{align*}
        I_{e+e_0+e_3}(I)^3 &\subseteq (I,\fq)^{[p^{e+e_0}]}I_{e+e_0+e_3}(I)\\
        &\subseteq (I,\fq)^{[p^{e+e_0}]}I_{e}(I) \mbox{ {\footnotesize by Lemma~\ref{lemma Some properties of splitting ideals} (\ref{splitting ideals some comparisons 3})}}\\
        &\subseteq I^{[p^{e+e_0}]} + \fq^{[p^{e+e_0}]}I_e(I)\\
        &\subseteq I^{[p^e]} \mbox{ {\footnotesize by (\ref{equation multiplier simple})}}.
    \end{align*}
    Clearly $p^2\geq 3$ and $I_{e+e_0+e_3}(I)^{[p^2]}\subseteq I_{e+e_0+e_3}(I)^{p^2}\subseteq I_{e+e_0+e_3}(I)^3$. By the above, for all $e\in\NN$,
    \[
    I_{e+e_0+e_3+2}(I)^{[p^2]}\subseteq I^{[p^{e+2}]}.
    \]
    The ring $R$ is $F$-pure, implying that for all $e\in\NN$,
    \[
    I_{e+e_0+e_3+2}(I)\subseteq I^{[p^e]}.
    \]
    The constant $e_0+e_3+2$ is independent of the ideal $I$. This completes the proof of (\ref{Main thm 2 special}).

    To prove Claim (\ref{Main thm 1 special}) we let $e'\in\NN$ as in (\ref{Main thm 2 special}), i.e., for all ideals $I\subseteq R$ and $e\in\NN$, $I_{e+e'}(I)\subseteq I^{[p^e]}$. By Lemma~\ref{lemma Some properties of splitting ideals} (\ref{splitting ideals and localization}), (\ref{splitting ideal containment}), and (\ref{associated primes of splitting ideals}),
    \[
    I^{\left([p^{e+e'}]\right)}\subseteq I_{e+e'}(I)\subseteq I^{[p^e]}.
    \]
    The assertion $R$ enjoys the Uniform Symbolic Topology Property follows by Lemma~\ref{lemma method of attack to USTP}~(\ref{criteria for isolated non F reg}).
\end{proof}

%% file: SFRUSTPnew.tex
\section{Uniform Linear Equivalence of Topologies}\label{Section Main Theorems}

The following lemma contains a method of identifying a constant $C$ that can be used to linearly compare the symbolic powers with ordinary powers of an ideal $I$ belonging to an $F$-finite domain $R$. The method aligns with a technique used in Hochster and Huneke's proof that a non-singular ring of prime characteristic enjoys the Uniform Symbolic Topology Property, cf. \cite[Last paragraph of introduction]{HHComparison}.

\begin{lemma}
    \label{lemma method of attack to USTP}
    Let $R$ be an $F$-finite domain of Krull dimension $d$ and $I\subseteq R$.
    \begin{enumerate}
        \item\label{criteria for isolated non F reg} If there exist constants $t,e_0\in\NN$ so that for every $e\in \NN$,
        \[
        \left(I^{\left([p^{e+e_0}]\right)}\right)^t\subseteq I^{[p^e]},
        \]
        then there exists a constant $C$, depending only on $d,t,e_0$, a uniform Brian\c{c}on-Skoda bound of $R$, and a uniform symbolic multiplier constant of a uniform symbolic multiplier of $R$, so that for all $n\in\NN$, $I^{(Cn)}\subseteq I^n$. 
        \item\label{criteria for F regular} If $R$ is $F$-pure and if there exist constants $t,e_0\in\NN$ so that for every $e\in \NN$,
        \[
        \left(I_{e+e_0}(I)\right)^t\subseteq I^{[p^e]},
        \]
        then there exists a constant $e'\in\NN$, depending only on $e_0$ and $t$, so that for every $e\in\NN$, 
        \[
        I_{e+e'}(I)\subseteq I^{[p^e]}.
        \]
    \end{enumerate}

\end{lemma}

\begin{proof}
    For (\ref{criteria for isolated non F reg}), let $z\in R$ be a uniform symbolic multiplier of $R$ with uniform symbolic multiplier constant $C_0$, see Definition~\ref{definition uniform multiplier} and Lemma~\ref{lemma uniform Frobenius multipliers}. Let $C_1$ be as in Lemma~\ref{lemma uniform symbolic Frobenius power multipliers} so that for every ideal $\fa\subseteq R$ and $e\in \NN$, $\fa^{(C_1p^e)}\subseteq \fa^{\left([p^e]\right)}$. The constant $C_1$ only depends on $d$ and a Uniform Brian\c{c}on-Skoda bound of $R$. Let $f\in I^{(C_0tC_1p^{e_0}n)}$. For each $e\in \NN$, write $p^e = a_en +r_e$ with $0\leq r_e< n$. Fix a nonzero element $c\in I^{(C_0tC_1p^{e_0}n)}$. Then $cf^{a_e}\in I^{(C_0tC_1p^{e+e_0})}$ and
    \[
    z^tcf^{a_e}\in \left(I^{(C_1p^{e+e_0})}\right)^{t}\subseteq \left(I^{\left([p^{e+e_0}]\right)}\right)^t\subseteq I^{[p^e]}.
    \]
    Raising the containment to the $n$ and multiplying by $f^{r_e}$ implies that for all $e\in\NN$,
    \[
    z^{tn}c^nf^{p^e}\in \left(I^n\right)^{[p^e]}\subseteq (I^n)^{p^e}.
    \]
    The constants $t$ and $n$ are independent of $e$. Therefore $f\in \overline{I^n}$ by a theorem of Rees, see \cite[Corollary~6.8.12]{SwansonHuneke}. Hence for every $n\in\NN$, $I^{(C_0tC_1p^{e_0}n)}\subseteq \overline{I^n}$. If $B$ is a Uniform Brian\c{c}on-Skoda Bound of $R$, then for all $n\in\NN$,
    \[
    I^{(C_0tC_1p^{e_0}(B+1)n)}\subseteq \overline{I^{(B+1)n}}\subseteq I^n.
    \]

    For (\ref{criteria for F regular}), choose constant $e_1\in\NN$ so that $p^{e_1}\geq t$. If $f\in I_{e+e_0}(I)$, then $f^{p^{e_1}}\in \left(I_{e+e_0}(I)\right)^t\subseteq I^{[p^e]}$. The assumption $R$ is $F$-pure implies $f\in I^{[p^{e-e_1}]}$. Therefore for every $e\in\NN$, $I_{e+e_0+e_1}(I)\subseteq I^{[p^e]}$.
\end{proof}

The next theorem is the technical part of the paper and accomplishes two objectives in the class of $F$-finite domains:
\begin{enumerate}
    \item \label{objective 1 of main tech thm} If $R$ is an $F$-finite domain so that the non-strongly $F$-regular locus consists of isolated points, then $R$ enjoys the Uniform Symbolic Topology Property.
    \item\label{objective 2 of main tech thm} Proof of Main Theorem~\ref{Main Theorem USTP for splitting ideals}: If $R$ is an $F$-finite strongly $F$-regular domain, then there exists $e'\in\NN$ so that for all ideals $I\subseteq R$, $I_{e+e'}(I)\subseteq I^{[p^e]}$.
\end{enumerate}
If $R$ is an $F$-finite $F$-pure domain, $I\subseteq R$ an ideal, and $e\in\NN$, then $I^{([p^e])}\subseteq I_e(I)$, see Lemma~\ref{lemma Some properties of splitting ideals}~(\ref{splitting ideals some comparisons 1}). Therefore the property described by (\ref{objective 2 of main tech thm}) implies the Uniform Symbolic Topology Property, see Lemma~\ref{lemma method of attack to USTP}.

Proposition~\ref{proposition ustp sfr restricted assumptions} is an indication of an inductive technique being used in the proof of Theorem~\ref{theorem USTP in isolated non SFR rings}. The technique is built upon an adaptation of a singular analogue of the non-singular method of test/multiplier ideal methods used in the proof of the Uniform Symbolic Topology Property of non-singular rings found in \cite{ELS, MaSchwedeSymbolic, TakagiYoshida, MurayamaSymbolic}. More specifically, methods built upon subadditivity results of multiplier/test ideals in non-singular rings, that are non-applicable to singular rings, is replaced by finding a sequence of ideals $\JJ\subseteq (\JJ,c_1)\subseteq (\JJ,c_1,c_2)\subseteq \cdots (\JJ,c_1,c_2,\ldots,c_t)$ and descriptions of uniformity relative to each ideal of the chain that are used to have multiplier type effects that eventually allow us to invoke Lemma~\ref{lemma method of attack to USTP}.

The objective of (\ref{objective 2 of main tech thm}) is of strong significance when restricted to class of $\fm$-primary ideals of a local strongly $F$-regular ring $(R,\fm,k)$. Even when restricted further to the single ideal $\fm$. The existence of an $e'\in\NN$ so that for all $e\in\NN$, $I_{e+e'}(\fm)\subseteq \fm^{[p^e]}$ is an important result that implies the non-trivial implication that a local strongly $F$-regular ring is strongly $F$-regular if and only if the $F$-signature of $R$ is positive, \cite[Main Result]{AberbachLeuschke}.

\begin{theorem}
    \label{theorem USTP in isolated non SFR rings}
    Let $R$ be an $F$-finite domain of prime characteristic $p>0$.
    \begin{enumerate}
        \item\label{part of isolated nsfr} If $R_\fp$ is strongly $F$-regular for all non-maximal primes $\fp$ of $\Spec(R)$, then $R$ enjoys the Uniform Symbolic Topology Property.
        \item\label{part sfr and splitting ideals} If $R$ is strongly $F$-regular, not only does $R$ enjoy the Uniform Symbolic Topology Property, there exists a constant $e'\in\NN$ so that for all ideals $I\subseteq R$ and $e\in\NN$, 
        \[
        I_{e+e'}(I)\subseteq I^{[p^e]}.
        \]
    \end{enumerate}
\end{theorem}

\begin{proof}
    We begin with a reduction to simplify our arguments. An $F$-finite ring has finite Krull dimension. Therefore there are finitely many values $d_1,\ldots,d_t\in\NN$ so that if $\fm\in\Max(R)$ then $\dim(R_{\fm})\in \{d_1,d_2,\ldots,d_t\}$. Let $W_{d_i}\subseteq \Spec(R)$ be the scheme of prime ideals $\fp\in \Spec(R)$ so that there exists some $\fm\in\Max(R)$ so that $\dim(R_{\fm})=d_i$ and $\fp\subseteq \fm$. Each of the schemes $W_{d_i}$ is covered by finitely many affine schemes that enjoy the hypotheses of the theorem. Ideal containments can be checked locally. We therefore can work on the finitely many affine pieces of the finitely many schemes $W_{d_i}$ and assume there exists $d\in\NN$ so that for all $\fm\in\Max(R)$, $\dim(R_{\fm})=d$.
    
    The content of the theorem is already know if $R$ is non-singular: If $R$ is non-singular then $R$ enjoys the Uniform Symbolic Topology Property by \cite[Theorem~1.1]{HHComparison}. When $R$ is non-singular, then for all ideals $I\subseteq R$ and $e\in\NN$, $I_e(I)=I^{[p^e]}$, see Lemma~\ref{lemma Some properties of splitting ideals} (\ref{splitting ideals regular}). 
    
    Assume that $R$ is singular and let $\JJ$ be the reduced ideal defining the singular locus of $\Spec(R)$. By Theorem~\ref{Theorem Splitting ideals and powers of ideals} there exists constant $e_0\in\NN$ so that for all ideals $\fa\subseteq R$, for all $e\in\NN$, $I_e(\fa^{[p^{e_0}]})\subseteq \fa^{[p^e]}$.

    \begin{claim*}\label{claim that adopts chevalley plus multiplier proof}
        Let $\fq\in\Ass(\JJ)$. There exists constants $D_{\fq},t_{\fq},e_{1,\fq},e_{2,\fq}\in\NN$ so that for all ideals $I\subseteq R$, $e\in \NN$, and maximal ideals $\fm$ satisfying either $\fm\not\in\Ass(I)$ or $R_\fm$ is strongly $F$-regular,
            \[
            I_{e+e_{2,\fq}}(I_{e_{1,\fq}}(I^{(D_\fq)}))^{t_\fq}R_\fm\subseteq (I,\fq)^{[p^e]}R_\fm.
            \]
        If $R$ is strongly $F$-regular, then $t_{\fq},e_{1,\fq}$, and $e_{2,\fq}$ can be chosen so that for all $e\in\NN$,
        \[
        I_{e+e_{2,\fq}}(I_{e_{1,\fq}}(I))^{t_\fq}\subseteq (I,\fq)^{[p^e]}.
        \]  
    \end{claim*}

    \begin{remark}
        If $I\subseteq R$ is an ideal and $\fm$ is a maximal ideal, then the associated primes of $IR_\fm$ are the associated primes of $I$ that are contained in $\fm$. Therefore if $R_\fm$ is either strongly $F$-regular or $\fm\not\in\Ass(I)$, then Lemma~\ref{lemma Some properties of splitting ideals} (\ref{associated primes of splitting ideals}), (\ref{splitting ideals some comparisons 1}), and (\ref{splitting ideals some comparisons 3}) are applicable to the localized ideal $IR_\fm$ as $R_\fp$ is assumed to be strongly $F$-regular, hence $F$-pure, for non-maximal primes $\fp$.
    \end{remark}

    \begin{proof}[Proof of Claim]
    If $\fq$ is a maximal ideal, then for each ideal $I\subseteq R$ either $I\subseteq \fq$ or $(I,\fq)=R$. Therefore the claim is already a consequence of Lemma~\ref{Uniform Chevalley Lemma for Frobenius powers} if $\fq$ is maximal. Assume that $\fq$ is non-maximal.
    
    Let $h=\height(\fq)$. We inductively construct finite-length lists of prime ideals $\Gamma_{i, \fq}$ of height $h\leq i\leq d$ (allowing for the possibility of repeated entries in $\Gamma_{i, \fq}$), with additional properties as described below. A defining feature of the lists $\Gamma_{i,\fq}$ is that for each $h+1\leq i\leq d$, the list $\Gamma_{i,\fq}$ is partitioned by the primes belonging to the list $\Gamma_{i-1,\fq}$.
        \begin{enumerate}[(i)]
            \item\label{height h ideal} $\Gamma_{h,\fq} = \{\fq\}$.
            \item\label{splitting ideal containmen with correct Frobenius power} If $i\geq h$ and the list $\Gamma_{i,\fq}$ has been constructed, then for each $\fp\in \Gamma_{i,\fq}$ we let $e_{\fp}\in\NN$ be chosen so that $I_{e_{\fp}}(\fp)\subseteq \fp^{[p^{e_0}]}$. Such constants exist by Lemma~\ref{lemma standard linear comparision of splitting ideals and Frobenius Powers} (\ref{standard lin comparison F reg}).
            \item\label{multipliers for the primes in the list} If $h\leq i<d$ and the list $\Gamma_{i,\fq}$ has been constructed, then for each prime $\fp\in\Gamma_{i,\fq}$, let $c_\fp\in R\setminus\fp$ be chosen as in Lemma~\ref{lemma multiplier for singular locus prime} so for all ideals $I\subseteq R$ and $e\in \NN$,
            \[
            c_\fp^{p^e}I_e(I_{e_{\fp}}(I))\subseteq (I,\fp)^{[p^e]}.
            \]
            If necessary, $c_\fp$ can be multiplied by a non-unit of $R$ so that $c_\fp$ is a non-unit and the primes in $\Ass\left(\sqrt{(\fp,c_\fp)}\right)$ are of height $1$ larger than the height of $\fp$ by Krull's Principal Ideal Theorem. 
            \item\label{inductive construction of the lists} If $h+1\leq i\leq d$ then $\Gamma_{i,\fq} = \bigcup_{\fp\in \Gamma_{i-1,\fq}\setminus \Max(R)}\Ass\left(\sqrt{(\fp,c_\fp)}\right)$. A prime $\fp'$ is repeated in the list $\Gamma_{i,\fq}$ every time $\fp'$ appears as an associated prime of $\sqrt{(\fp,c_\fp)}$ and $\fp\in \Gamma_{i-1,\fq}$.
            \item\label{large enough product of primes recovers multipliers} For each $h\leq i\leq d$, $\fp\in \Gamma_{i,\fq}$, $e'_{\fp}\in\NN$ is so that $\prod_{\fp'\in\Ass\left(\sqrt{(\fp,c_\fp)}\right)}(\fp')^{[p^{e'_{\fp}}]}\subseteq (\fp,c_\fp)$. In particular, for all $e\in\NN$,
            \[
            \prod_{\fp'\in\Ass\left(\sqrt{(\fp,c_\fp)}\right)}(\fp')^{[p^{e+e'_{\fp}}]}\subseteq (\fp,c_\fp)^{[p^e]}.
            \]

            \item\label{the constant e_1}  Let $e_{1,\fq} = \max\{e_\fp\mid \fp\in\Gamma_{i,\fq}\mbox{ and }i\leq h\leq d-1\}$.
            \begin{remark}
                Suppose that $I\subseteq R$ is an ideal and $\fm$ is a maximal ideal. If either $\fm\not\in\Ass(I)$ or if $R_\fm$ is strongly $F$-regular, then by Lemma~\ref{lemma Some properties of splitting ideals}  (\ref{splitting ideal containment}) and (\ref{splitting ideals some comparisons 3}), for all $e\in \NN$,
                \[
                I_{e}(I_{e_{1,\fq}}(I))R_\fm\subseteq I_{e}(I_{e_\fp}(I))R_\fm.
                \]
            \end{remark}
            
            \item\label{the constant e_2} Let $e_2 = \max\{e_{\fp}'\mid \fp\in \Gamma_{i,\fq}\mbox{ and }h\leq i\leq d\}$.

            \item\label{the constant e3} By Lemma~\ref{Uniform Chevalley Lemma for Frobenius powers} (\ref{Uniform Chevalley non-F-reg}) there exists constants $D_\fq,e_3\in\NN$, so that for each of the finitely many maximal ideals $\fm$ minimal over the finitely many ideals of the form $(\fp,c_\fp)$ for some $\fp\in\Gamma_{d-1,\fq}$, for every ideal $I\subseteq \fm$ so that either $\fm\not\in\Ass(I)$ or $R_\fm$ is strongly $F$-regular, for every $e\in \NN$, 
            \begin{align*}
            I_{e+e_3}(I^{(D_\fq)})R_\fm\subseteq (\fp,c_{\fp})^{[p^e]}R_\fm.
            \end{align*}
            By Lemma~\ref{Uniform Chevalley Lemma for Frobenius powers} (\ref{Uniform Chevalley F-reg}), if $R$ is strongly $F$-regular the constant $D_\fq$ can be chosen to be $1$.
        \end{enumerate}

        We now proceed with the proof of the claim. We are assuming $\fq$ is not a maximal ideal. Let $I\subseteq R$ be an ideal. The content of the claim is trivial with respect to a maximal ideal $\fm$ if $(I,\fq)\not\subseteq \fm$. In what follows, an inductive argument is used to show that for each maximal ideal $\fm\in\Max(R)$ containing $(I,\fq)$ so that either $\fm\not\in\Ass(I)$ or if $R_\fm$ is strongly $F$-regular, there exists constants $s_\fm,e_\fm\in\NN$, independent of the maximal ideal $\fm$ and the ideal $I$, so that for all $e\in\NN$,
        \begin{equation}\label{goal of inductive argument}
        I_{e+e_\fm}(I^{(D_\fq)})^{s_\fm}R_\fm\subseteq (I,\fq)^{[p^e]}R_\fm.
        \end{equation}
        To this end we fix a maximal ideal $\fm\supseteq (I,\fq)$ so that either $\fm\not\in\Ass(I)$ or $R_\fm$ is strongly $F$-regular. For each $\fp\in \Gamma_{d-1,\fq}$, the ideal $(\fp,c_\fp)R_\fm$ is either $\fm$-primary or the unit ideal of $R_\fm$. Either scenario, property~(\ref{the constant e3}) implies that for every $\fp\in \Gamma_{d-1,\fq}$ and $e\in\NN$, 
        \begin{align}\label{contain by property constant e3}
        I_{e+e_3}(I_{e_{1,\fq}}(I^{(D_\fq)}))R_\fm\subseteq (\fp,c_\fp)^{[p^e]}R_\fm.
        \end{align}
        Let $|\Gamma_{d-1,\fq}|$ denote the length of the list $\Gamma_{d-1,\fq}$. Then for every $e\in \NN$,
        \begin{align*}
            I_{e+e_3}(I_{e_{1,\fq}}(I^{(D_\fq)}))^{|\Gamma_{d-1,\fq}|}R_\fm &=\left(\prod_{\fp \in \Gamma_{d-1,\fq}}I_{e+e_3}(I_{e_{1,\fq}}(I^{(D_\fq)}))\right)R_\fm\\
            &\subseteq \left(\prod_{\fp \in \Gamma_{d-1,\fq}} (\fp,c_\fp)^{[p^e]}\right)R_\fm \mbox{ {\footnotesize by (\ref{contain by property constant e3}).}}
        \end{align*}
        Multiplying the containment by $I_{e+e_3}(I_{e_{1,\fq}}(I^{(D_\fq)}))^{|\Gamma_{d-1,\fq}|}$,
        \begin{align*}
            I_{e+e_3}(I_{e_{1,\fq}}(I^{(D_\fq)}))^{2|\Gamma_{d-1,\fq}|}R_\fm &\subseteq \left(\prod_{\fp \in \Gamma_{d-1,\fq}} (\fp,c_\fp)^{[p^e]}\right) I_{e+e_3}(I_{e_{1,\fq}}(I^{(D_\fq)}))^{|\Gamma_{d-1,\fq}|}R_\fm\\
            &= \left(\prod_{\fp \in \Gamma_{d-1,\fq}} (\fp,c_\fp)^{[p^e]}I_{e+e_3}(I_{e_{1,\fq}}(I^{(D_\fq)}))\right)R_\fm\\
            &\subseteq \left(\prod_{\fp \in \Gamma_{d-1,\fq}} (\fp,c_\fp)^{[p^e]}I_{e}(I_{e_{1,\fq}}(I^{(D_\fq)}))\right)R_\fm\mbox{ {\footnotesize by Lemma~\ref{lemma Some properties of splitting ideals} (\ref{splitting ideals some comparisons 3})}}\\
            &\subseteq \left(\prod_{\fp \in \Gamma_{d-1,\fq}} (\fp,c_\fp)^{[p^e]}I_{e}(I_{e_\fp}(I^{(D_\fq)}))\right)R_\fm\mbox{ {\footnotesize by property (\ref{the constant e_1})}}\\
            &\subseteq \left(\prod_{\fp \in \Gamma_{d-1,\fq}} (\fp,c_\fp)^{[p^e]}I_{e}(I_{e_\fp}(I))\right)R_\fm\mbox{ {\footnotesize by Lemma~\ref{lemma Some properties of splitting ideals} (\ref{splitting ideal containment})}}\\
            &\subseteq \left(\prod_{\fp \in \Gamma_{d-1,\fq}} (I,\fp)^{[p^e]}\right)R_\fm \mbox{ {\footnotesize by property (\ref{multipliers for the primes in the list})}}\\
            &\subseteq \left(I^{[p^e]}, \prod_{\fp\in\Gamma_{d-1,\fq}}\fp^{[p^e]}\right)R_\fm.
        \end{align*}
        If $h=d-1$ then $\Gamma_{d-1,\fq}=\Gamma_{h,\fq}=\{\fq\}$. Therefore if $h=d-1$ then for all $e\in\NN$,
        \[
        I_{e+e_3}(I_{e_{1,\fq}}(I^{(D_\fq)}))^{|\Gamma_{d-1,\fq}|}R_\fm\subseteq \left(I^{[p^e]}, \fq^{[p^e]}\right)R_\fm.
        \]
        The constants $e_{1,\fq},e_3,|\Gamma_{d-1,\fq}|$ did not depend on the choice of maximal ideal $\fm$. Therefore if $h=d-1$, then we have accomplished the objective of the claim. Suppose that $h\leq d-2$. By construction, the primes appearing in the list $\Gamma_{d-1,\fq}$ are partitioned by the sets $\Ass(\sqrt{(\fp',c_{\fp'})})$ as $\fp'$ varies through non-maximal prime ideals of the list $\Gamma_{d-2,\fq}$. Therefore 
        \begin{equation}\label{equation multiplication of primes}
        \prod_{\fp\in\Gamma_{d-1,\fq}}\fp^{[p^e]}\subseteq \prod_{\fp'\in\Gamma_{d-2,\fq}}\left(\sqrt{(\fp',c_{\fp'})}\right)^{[p^{e}]}.
        \end{equation}
        Hence,        
        \begin{align*}
        I_{e+e_3}(I_{e_{1,\fq}}(I^{(D_\fq)}))^{2|\Gamma_{d-1,\fq}|}R_\fm &\subseteq \left(I^{[p^e]}, \prod_{\fp\in\Gamma_{d-1,\fq}}\fp^{[p^e]}\right)R_\fm \\
        &\subseteq \left(I^{[p^e]}, \prod_{\fp'\in\Gamma_{d-2,\fq}}\left(\sqrt{(\fp',c_{\fp'})}\right)^{[p^{e}]}\right)R_\fm \mbox{ {\footnotesize by (\ref{equation multiplication of primes})}} \\
        &\subseteq \left(I^{[p^e]}, \prod_{\fp'\in\Gamma_{d-2,\fq}}(\fp',c_{\fp'})^{[p^{e-e'_{\fp'}}]}\right)R_\fm \mbox{ {\footnotesize by property (\ref{large enough product of primes recovers multipliers})}}\\
        &\subseteq \left(I^{[p^e]}, \prod_{\fp'\in\Gamma_{d-2,\fq}}(\fp',c_{\fp'})^{[p^{e-e_2}]}\right)R_\fm \mbox{ {\footnotesize by property (\ref{the constant e_2})}}.
        \end{align*}

         Let $e_{d-2,\fq} = e_2+e_3$ and $s_{d-2}=2|\Gamma_{d-1,\fq}|$, a pair of constants only dependent on the height of the maximal ideal $\fm$. By the previous containments, for all $e\in \NN$,
        \[
        I_{e+e_{d-2,\fq}}(I_{e_{1,\fq}}(I^{(D_\fq)}))^{s_{d-2}}R_\fm\subseteq \left(I^{[p^e]}, \prod_{\fp\in\Gamma_{d-2, \fq}}(\fp,c_{\fp})^{[p^e]}\right)R_\fm.
        \]
    
        Inductively, we assume that for $h\leq i \leq d-2$ we have found constants $e_{i,\fq,\fm},s_{i,\fm}\in \NN$ independent of the maximal ideal $\fm$ and the ideal $I\subseteq R$, so that for all $e\in \NN$,
        \[
        I_{e+e_{i,\fq,\fm}}(I_{e_{1,\fq}}(I^{(D_\fq)}))^{s_{i,\fm}}R_\fm\subseteq \left(I^{[p^e]}, \prod_{\fp\in\Gamma_{i, \fq}}(\fp,c_{\fp})^{[p^e]}\right)R_\fm.
        \]
        Let $|\Gamma_{i, \fq}|$ be the length of the list $\Gamma_{i,\fq}$. Multiplying by $I_{e+e_{i,\fq,\fm}}(I_{e_{1,\fq}}(I^{(D_\fq)}))^{|\Gamma_{i, \fq}|}$ gives a containment of ideals
        \begin{align*}
            I_{e+e_{i,\fq,\fm}}(I_{e_{1,\fq}}(I^{(D_\fq)}))^{s_{i,\fm}+|\Gamma_{i, \fq}|}R_\fm &\subseteq \left(I^{[p^e]}, \prod_{\fp\in\Gamma_{i, \fq}}(\fp,c_{\fp})^{[p^e]}\right)I_{e+e_{i,\fq,\fm}}(I_{e_{1,\fq}}(I^{(D_\fq)}))^{|\Gamma_{i, \fq}|}R_\fm\\
            &\subseteq \left(I^{[p^e]}, \prod_{\fp\in\Gamma_{i, \fq}}(\fp,c_{\fp})^{[p^e]}\right)I_{e}(I_{e_{1,\fq}}(I^{(D_\fq)}))^{|\Gamma_{i, \fq}|}R_\fm \, \mbox{{\footnotesize by Lemma~\ref{lemma Some properties of splitting ideals} (\ref{splitting ideals some comparisons 3})}}\\
            &\subseteq \left(I^{[p^e]}, \prod_{\fp\in\Gamma_{i, \fq}}(\fp,c_{\fp})^{[p^e]}I_{e}(I_{e_{1,\fq}}(I^{(D_\fq)}))\right)R_\fm \\
            &\subseteq \left(I^{[p^e]}, \prod_{\fp\in\Gamma_{i, \fq}}(\fp,c_{\fp})^{[p^e]}I_{e}(I_{e_{1,\fq}}(I))\right)R_\fm \mbox{ {\footnotesize by Lemma~\ref{lemma Some properties of splitting ideals} (\ref{splitting ideal containment})}} \\
            &\subseteq \left(I^{[p^e]}, \prod_{\fp\in\Gamma_{i, \fq}}(\fp,c_{\fp})^{[p^e]}I_{e}(I_{e_\fp}(I))\right)R_\fm \mbox{ {\footnotesize by property (\ref{the constant e_1})}} \\
            &\subseteq \left(I^{[p^e]}, \prod_{\fp\in\Gamma_{i, \fq}}\fp^{[p^e]}\right)R_\fm  \mbox{ {\footnotesize by property (\ref{multipliers for the primes in the list}).}}
        \end{align*}
        If $i\geq h+1$ then the primes appearing in the list $\Gamma_{i,\fq}$ are partitioned by sets $\Ass(\sqrt{(\fp',c_{\fp'})})$ as $\fp'$ varies through non-maximal primes of the list $\Gamma_{i-1,\fq}$. Properties (\ref{large enough product of primes recovers multipliers}) and (\ref{the constant e_2}) then provide a containment of ideals
        \[
        I_{e+e_{i,\fq,\fm}}(I_{e_{1,\fq}}(I))^{s_{i,\fm}+|\Gamma_{i, \fq}|}R_\fm\subseteq \left(I^{[p^e]}, \prod_{\fp\in\Gamma_{i-1, \fq}}(\fp',c_{\fp'})^{[p^{e-e_2}]}\right)R_\fm.
        \]
        Therefore if $i\geq h+1$, then the constants $e_{i-1,\fq,\fm}=e_2+e_{i,\fq,\fm}$ and $s_{i-1,\fm}=s_{i,\fm}+|\Gamma_{i,\fq}|$ are only dependent on the height of the maximal ideal $\fm$ and have the property that for every $e\in \NN$,
        \[
        I_{e+e_{i-1,\fq,\fm}}(I_{e_{1,\fq}}(I))^{s_{i-1,\fm}}R_\fm\subseteq \left(I^{[p^e]}, \prod_{\fp\in\Gamma_{i-1, \fq}}(\fp,c_{\fp})^{[p^e]}\right)R_\fm.
        \]

        Recall that $\Gamma_{h,\fq}=\{\fq\}$. By the induction argument above, there are constants $e_{h,\fq,\fm},s_{h,\fm}$ independent of the maximal ideal $\fm$ and the ideal $I$ so that for all $e\in \NN$,
        \[
         I_{e+e_{h,\fq,\fm}}(I_{e_{1,\fq}}(I^{(D_\fq)}))^{s_{h,\fm}}R_\fm\subseteq \left(I^{[p^e]}, (\fq,c_{\fq})^{[p^e]}\right)R_\fm.
        \]
        For each ideal $I\subseteq R$ let
        \[
        \footnotesize
        e_{h,\fq,I} = \max\{e_{h,\fq,\fm}\mid \fm\in \Max(R)\mbox{ and }(I,\fq)\subseteq \fm, \mbox{ either } \fm\not\in\Ass(I) \mbox{ or } R_\fm \mbox{ is strongly } F\mbox{-regular}\}
        \]
        and 
        \[
        \footnotesize
        s_{h,I}=\max\{s_{h,\fm}\mid \fm\in \Max(R)\mbox{ and }(I,\fq)\subseteq \fm, \mbox{ either } \fm\not\in\Ass(I) \mbox{ or } R_\fm \mbox{ is strongly } F\mbox{-regular}\}.
        \]
        The constants $e_{h,\fq,I}$ and $s_{h,I}$ are independent of the ideal $I$. Let
        \[
        e_{2,\fq} = \max\{e_{h,\fq,I}\mid I\subseteq R\}\cup\{e_{1,\fq}\}
        \]
        and 
        \[
        t_{\fq}=\max\{s_{h, I}\mid I\subseteq R\}.
        \]
        Then for every ideal $I\subseteq R$, $e\in\NN$, and maximal ideal $\fm$ containing $(I,\fq)$ so that either $\fm\not\in\Ass(I)$ or $R_\fm$ is strongly $F$-regular,
        \begin{align*}
            I_{e+e_{2,\fq}}(I_{e_{1,\fq}}(I^{(D_\fq)}))^{t_\fq}R_\fm &\subseteq I_{e+e_{h,\fq,I}}(I_{e_{1,\fq}}(I^{(D_\fq)}))^{t_{\fq}}R_\fm\mbox{ {\footnotesize by Lemma~\ref{lemma Some properties of splitting ideals} (\ref{splitting ideals some comparisons 3})}}\\
            &\subseteq \left(I^{[p^e]}, (\fq,c_{\fq})^{[p^e]}\right)R_\fm.
        \end{align*}
        Multiplying the above containment by $I_{e+e_{2,\fq}}(I_{e_{1,\fq}}(I^{(D_\fq)}))$, 
        \begin{align*}
            I_{e+e_{2,\fq}}(I_{e_{1,\fq}}(I^{(D_\fq)}))^{t_{\fq}+1}R_\fm&\subseteq \left(I^{[p^e]}, (\fq,c_{\fq})^{[p^e]}\right)I_{e+e_{2,\fq}}(I_{e_{1,\fq}}(I^{(D_\fq)}))R_\fm\\
            &\subseteq \left(I^{[p^e]}, (\fq,c_{\fp})^{[p^e]}\right)I_{e}(I_{e_{1,\fq}}(I))R_\fm \mbox{ {\footnotesize by Lemma~\ref{lemma Some properties of splitting ideals}  (\ref{splitting ideal containment}) and (\ref{splitting ideals some comparisons 3})}}\\
            &\subseteq \left(I^{[p^e]}, \fq^{[p^e]}\right)R_\fm \mbox{ {\footnotesize property (\ref{multipliers for the primes in the list}).}}
        \end{align*}
        The constants $e_{2,\fq}$ and $s_h+1$ are independent of the choice of ideal $I\subseteq R$ and choice of a maximal ideal $\fm$ containing $(I,\fq)$ so that either $\fm\not\in\Ass(I)$ or $R_\fm$ is strongly $F$-regular. 
        
        If $R$ is strongly $F$-regular, then $D_\fq$ could have be chosen to be $1$ by property (\ref{the constant e3}). Moreover, there was no restriction on the choice of maximal ideal in the above argument. Therefore if $R$ is strongly $F$-regular, then for all ideals $I\subseteq R$ and $e\in\NN$,
        \[
        I_{e+e_{2,\fq}}(I_{e_{1,\fq}}(I))^{t_\fq+1}\subseteq \left(I^{[p^e]}, \fq^{[p^e]}\right).
        \]
        This completes the proof of the claim.
    \end{proof} 

    We continue the proof of the theorem and adopt the notation in the statement of Claim~\ref{claim that adopts chevalley plus multiplier proof}. Let $D=\max\{D_\fq\mid \fq\in\Ass(\JJ)\}$, $e_1=\max\{e_{1,\fq}\mid \fq\in\Ass(\JJ)\}$, and $e_{2}=\max\{e_{2,\fq}\mid \fq\in\Ass(\JJ)\}$. By Lemma~\ref{lemma uniform Frobenius multipliers}, there exists a constant $e_3$ so that for all ideals $I\subseteq R$ and $e\in \NN$
    \begin{align}\label{equation J multiplier a}
    \JJ^{[p^{e+e_3}]}I_e(I)\subseteq I^{[p^e]}.
    \end{align}

    We now prove that if $\Spec(R)$ has at worst isolated non-strongly $F$-regular points, then $R$ enjoys the Uniform Symbolic Topology Property. It suffices to find a constant $C$ so that for all ideals $I\subseteq R$ and maximal ideals $\fm$, $I^{(Cn)}R_\fm\subseteq I^nR_\fm$. If $\fm\in \Ass(I)$ then $I^{(n)}R_\fm = I^nR_\fm$. Assume that $\fm\not\in\Ass(I)$. Let $t = \sum_{\fq\in \Ass(\JJ)}t_{\fq}$ and $e_4\in\NN$ a constant so that $p^{e_4}\geq D$. By Lemma~\ref{lemma Some properties of splitting ideals} (\ref{splitting ideals containment of bracket powers}) and (\ref{associated primes of splitting ideals}),
    \begin{equation}
        \label{equation from 2 statements of splitting ideal lemma}
        I^{\left([p^{e+e_{1}+e_2+e_3+e_4}]\right)}\subseteq I_{e+e_1+e_2+e_3}(I^{\left([p^{e_4}]\right)}).
    \end{equation}
    Then for all $e\in\NN$,
    \begin{align*}
        \left(I^{\left([p^{e+e_{1}+e_2+e_3+e_4}]\right)}\right)^{t+1}R_\fm &\subseteq \left(I_{e+e_1+e_2+e_3}(I^{\left([p^{e_4}]\right)})\right)^{t+1}R_\fm \mbox{ {\footnotesize by (\ref{equation from 2 statements of splitting ideal lemma})}}\\
        &\subseteq \left(I_{e+e_1+e_2+e_3}(I^{(D)})\right)^{t+1}R_\fm \mbox{ {\footnotesize by Lemma~\ref{lemma Some properties of splitting ideals} (\ref{splitting ideal containment})}}\\
        &=\left(\prod_{\fq\in\Ass(\JJ)}I_{e+e_1+e_2+e_3}(I^{(D)})^{t_\fq}\right)I_{e+e_1+e_2+e_3}(I^{(D)})R_\fm\\
        &\subseteq\left(\prod_{\fq\in\Ass(\JJ)}I_{e+e_1+e_2+e_3}(I^{(D_\fq)})^{t_\fq}\right)I_{e+e_1+e_2+e_3}(I)R_\fm \mbox{ {\footnotesize by Lemma~\ref{lemma Some properties of splitting ideals} (\ref{splitting ideal containment})}}\\
        &\subseteq \left(\prod_{\fq\in\Ass(\JJ)}I_{e+e_2+e_3}(I_{e_1}(I^{(D_\fq)}))^{t_\fq}\right)I_{e+e_1+e_2+e_3}(I)R_\fm \mbox{ {\footnotesize by Lemma~\ref{lemma Some properties of splitting ideals} (\ref{splitting ideals some comparisons 2})}}\\
        &\subseteq \left(\prod_{\fq\in\Ass(\JJ)}I_{e+e_2+e_3}(I_{e_1}(I^{(D_\fq)}))^{t_\fq}\right)I_{e}(I)R_\fm \mbox{ {\footnotesize by Lemma~\ref{lemma Some properties of splitting ideals} (\ref{splitting ideals some comparisons 3})}}\\
        &\subseteq \left(\prod_{\fq\in\Ass(\JJ)}I_{e+e_2+e_3}(I_{e_{1,\fq}}(I^{(D_\fq)}))^{t_\fq}\right)I_{e}(I)R_\fm \mbox{ {\footnotesize by Lemma~\ref{lemma Some properties of splitting ideals} (\ref{splitting ideals some comparisons 3})}}\\
        &\subseteq \left(\prod_{\fq\in\Ass(\JJ)}I_{e+e_{2,\fq}+e_3}(I_{e_{1,\fq}}(I^{(D_\fq)}))^{t_\fq}\right)I_{e}(I)R_\fm \mbox{ {\footnotesize by Lemma~\ref{lemma Some properties of splitting ideals} (\ref{splitting ideals some comparisons 3})}}\\
        &\subseteq \left(\prod_{\fq\in\Ass(\JJ)}(I,\fq)^{[p^{e+e_3}]}\right)I_e(I)R_\fm \mbox{ \footnotesize by Claim~\ref{claim that adopts chevalley plus multiplier proof}}\\
        &\subseteq \left(I^{[p^{e+e_3}]},\prod_{\fq\in\Ass(\JJ)}\fq^{[p^{e+e_3}]}\right)I_e(I) R_\fm\\
        &\subseteq (I,\JJ)^{[p^{e+e_3}]}I_e(I)R_\fm\\
        &\subseteq I^{[p^e]}R_\fm \mbox{ {\footnotesize by (\ref{equation J multiplier a}).}}
    \end{align*}
    The constants $e_1+e_2+e_3+e_4$ and $t+1$ were independent of the ideal $I$. By Lemma~\ref{lemma method of attack to USTP} (\ref{criteria for isolated non F reg}) there exists a constant $C$, depending only on the dimension of $R_\fm$, a uniform Brian\c{c}on-Skoda bound of $R_{\fm}$, and a uniform symbolic multiplier constant of a uniform symbolic multiplier of $R_\fm$, so that $I^{(Cn)}R_\fm\subseteq I^nR_\fm$ for all $n\in\NN$. Rings that are $F$-finite have finite Krull dimension. Integral closure of ideals commutes with localization, therefore a uniform Brian\c{c}on-Skoda bound of $R$ is a uniform Brian\c{c}on-Skoda Bound of $R_\fm$. The ring $R$ admits a uniform symbolic multiplier $z$ with a uniform symbolic multiplier constant $C_0$ by Lemma~\ref{lemma uniform Frobenius multipliers}. Symbolic powers and ordinary powers of ideals commute with localization. Hence $z$ is a uniform symbolic multiplier of $R_\fm$ with uniform symbolic multiplier constant $C_0$. Therefore there exists a constant $C$, not depending on the ideal $I$ nor the choice of maximal ideal $\fm$, so that for all $n\in\NN$, $I^{(Cn)}R_\fm\subseteq I^nR_\fm$. Therefore $R$ enjoys the Uniform Symbolic Topology Property.

    Now assume $R$ is strongly $F$-regular and continue to adopt the notation of Claim~\ref{claim that adopts chevalley plus multiplier proof}. Let $e_{1}=\max\{e_{1,\fq}\mid \fq\in\Ass(\JJ)\}$, $e_{2}=\max\{e_{2,\fq}\mid \fq\in\Ass(\JJ)\}$, and $t = \sum_{\fq\in \Ass(\JJ)}t_{\fq}$. Then for all ideals $I\subseteq R$ and $e\in \NN$,
    \begin{align*}
        \left(I_{e+e_1+e_2+e_3}(I)\right)^{t+1} &=\left(\prod_{\fq\in\Ass(\JJ)}I_{e+e_1+e_2+e_3}(I)^{t_\fq}\right)I_{e+e_1+e_2+e_3}(I)\\
        &\subseteq \left(\prod_{\fq\in\Ass(\JJ)}I_{e+e_2+e_3}(I_{e_1}(I))^{t_\fq}\right)I_{e+e_1+e_2+e_3}(I) \mbox{ {\footnotesize by Lemma~\ref{lemma Some properties of splitting ideals} (\ref{splitting ideals some comparisons 2})}}\\
        &\subseteq \left(\prod_{\fq\in\Ass(\JJ)}I_{e+e_2+e_3}(I_{e_1}(I))^{t_\fq}\right)I_{e}(I) \mbox{ {\footnotesize by Lemma~\ref{lemma Some properties of splitting ideals} (\ref{splitting ideals some comparisons 3})}}\\
        &\subseteq \left(\prod_{\fq\in\Ass(\JJ)}I_{e+e_2+e_3}(I_{e_{1,\fq}}(I))^{t_\fq}\right)I_{e}(I) \mbox{ {\footnotesize by Lemma~\ref{lemma Some properties of splitting ideals} (\ref{splitting ideals some comparisons 3})}}\\
        &\subseteq \left(\prod_{\fq\in\Ass(\JJ)}I_{e+e_{2,\fq}+e_3}(I_{e_{1,\fq}}(I))^{t_\fq}\right)I_{e}(I) \mbox{ {\footnotesize by Lemma~\ref{lemma Some properties of splitting ideals} (\ref{splitting ideals some comparisons 3})}}\\
        &\subseteq \left(\prod_{\fq\in\Ass(\JJ)}(I,\fq)^{[p^{e+e_3}]}\right)I_e(I) \mbox{ \footnotesize by Claim~\ref{claim that adopts chevalley plus multiplier proof}}\\
        &\subseteq (I,\JJ)^{[p^{e+e_3}]}I_e(I)\\
        &\subseteq I^{[p^e]} \mbox{ {\footnotesize by (\ref{equation J multiplier a}).}}
    \end{align*}
    The constants $e_1+e_2+e_3$ and $t+1$ were independent of the ideal $I$. By Lemma~\ref{lemma method of attack to USTP} (\ref{criteria for F regular}) there exists a constant $e'\in\NN$ so that for all ideals $I\subseteq R$ and $e\in\NN$, $I_{e+e'}(I)\subseteq I^{[p^e]}$.
\end{proof}

\begin{corollary}
    \label{corollary USTP non-F-finite rings}
    Let $R$ be a domain essentially of finite type over an excellent local ring $(A,\fm,k)$ of prime characteristic $p>0$. If $R_\fp$ is strongly $F$-regular for all non-maximal primes $\fp$ of $\Spec(R)$, then $R$ enjoys the Uniform Symbolic Topology Property.
\end{corollary}

\begin{proof}
    It is known that if $R \to S$ is a faithfully flat map and $I \subseteq R$ is an ideal, then $I^{(n)}S \cap R = I^{(n)}$ and $I^nS \cap R = I^n$. See \cite[Argument of Lemma 2.4, Step 3]{MurayamaSymbolic} for the necessary details. In particular, if a faithfully flat extension of $R$ enjoys the Uniform Symbolic Topology Property, then so does $R$. 
    
    We can replace $R$ with $R \otimes_A \widehat{A}$ and assume that $R$ is essentially of finite type over a complete local ring $(A, \mathfrak{m}, k)$. By \cite[Theorem~3.4 (ii)]{MurayamaGamma}, there exists a purely inseparable and faithfully flat ring map $R \to R^\Gamma$ such that $R^\Gamma$ is $F$-finite with the additional property that for all $\mathfrak{p} \in \Spec(R)$, $R_\mathfrak{p}$ is strongly $F$-regular if and only if $R^\Gamma_{\sqrt{\mathfrak{p} R^\Gamma}}$ is strongly $F$-regular. Consequently, $R$ satisfies the Uniform Symbolic Topology Property by a repeated application of \cite[Argument of Lemma 2.4, Step 3]{MurayamaSymbolic} and Theorem~\ref{theorem USTP in isolated non SFR rings} (\ref{part of isolated nsfr}).
\end{proof}

%% file: SFR_finite_Extensions.tex
\section{The Uniform Symbolic Topology Property and Finite Extensions}\label{Section Symbolic Powers, Intersections, finite extensions}

Throughout this section $R$ is assumed to be an excellent Noetherian ring of arbitrary characteristic and of finite Krull dimension. Progress on the Uniform Symbolic Topology Problem found in \cite{HKVfinite, HKAbelian, HKHypersurface} is on the behavior of symbolic powers of prime ideals. A Noetherian ring $R$ is said to enjoy the \emph{Uniform Symbolic Topology Property for prime ideals} if there exists a constant $C$ so that for all prime ideals $\fp\in\Spec(R)$ and $n\in\NN$, $\fp^{(C\height(\fp)n)}\subseteq \fp^n$. Huneke, Katz, and Validashti proved under suitable hypotheses that if $R\to S$ is a finite extension of domains, if $S$ enjoys the Uniform Symbolic Topology Property for prime ideals, then $R$ enjoys the Uniform Symbolic Topology Property for prime ideals, \cite[Corollary]{HKVfinite}. We extend their theorem from the study of prime ideals to all ideals. We first notice that study of the Uniform Symbolic Topology Property is readily reduced to the study of symbolic powers of integrally closed ideals.

\begin{lemma}
    \label{lem: USTP reduces to integrally closed ideals}
    Let $R$ be an excellent Noetherian domain that enjoys the Uniform Brian\c{c}on-Skoda Property. If there exists a constant $C$ so that for all integrally closed ideals $J\subseteq R$ and $n\in \NN$, $J^{(Cn)}\subseteq J^n$, then $R$ enjoys the Uniform Symbolic Topology Property.
\end{lemma}

\begin{proof}
    Let $B$ be a Uniform Brian\c{c}on-Skoda Bound of $R$, $I\subseteq R$ an ideal, $W$ the complement of the union of the associated primes of $I$, and $J = \overline{I}R_W\cap R$. There exists an element $c\in R$ so that $J = (\overline{I}:_Rc)$, hence $J$ is integrally closed, see discussion of \cite[Remark~1.3.2 (2)]{SwansonHuneke}, and for all $n\in\NN$, $J^nR_W\cap R\subseteq J^{(n)}$. Integral closure commutes with localization. So for each $n\in\NN$, there is a containment of ideals 
    \[
    I^{(n+B)}\subseteq \overline{J^{n+B}}R_W\cap R\subseteq J^nR_W\cap R\subseteq J^{(n)}.
    \]
    Therefore for each $n\in \NN$,
    \[
    I^{(C(B+1)^2)n}\subseteq \overline{J^{C(B+1)^2n}}R_W\cap R\subseteq J^{C(B+1)n+B}R_W\cap R\subseteq J^{(C(B+1)n)}\subseteq J^{(B+1)n}\subseteq J^{n+B}\subseteq I^n.
    \]
    The constant $C(B+1)^2$ is independent of the ideal $I\subseteq R$ and $n\in\NN$.
\end{proof}

We utilize the theory of Rees valuations to study the powers of an ideal. For an introduction to theory of Rees valuations see \cite[Chapter~10]{SwansonHuneke}. The following summarizes the necessary information.

\begin{remark}[{\cite[Chapter~10]{SwansonHuneke}}]
    \label{remark Rees valuations}
    Let $R$ be an excellent Noetherian domain, $K$ the fraction field of $R$, and $I\subseteq R$ an ideal. Under the given hypotheses, the \emph{Rees valuations} of $I$ are the discrete valuations $\nu:K^{\times}\to\ZZ$ corresponding to the exceptional components of the normalized blowup of the ideal $I$. Let $\overline{B}_I\xrightarrow{\pi_I} \Spec(R)$ denote the normalized blowup of $I$ and $\Exc(\pi_I)$ be the finite set of irreducible components of the exceptional locus of $\pi_I$. Each component $E\in \Exc(\pi_I)$ has codimension $1$ and the local ring $\O_E$ is a discrete valuation ring. For each $E\in\Exc(\pi_I)$ let $\fm_E$ be the maximal ideal of $\O_E$ and $\nu_E$ the corresponding discrete valuation. If $J\subseteq R$ then $\nu_E(J) := \min\{\nu_E(x)\mid x\in J\}$ and is the exponent so that $J\OO_E = \fm_E^{\nu_E(J)}\OO_E$. For every ideal $J\subseteq R$ and $n\in\NN$, $J\subseteq \overline{I^n}$ if and only if $\nu_E(J)\geq n\nu_E(I)$ for all $E\in\Exc(\pi_I)$.
\end{remark}

Finite extensions of Noetherian domains $R\to S$ are \emph{integral}, every element $s\in S$ is the root of a monic polynomial with coefficients in $R$. When $R\to S$ is integral and $I\subseteq R$ is an ideal, then the integral closure of $I$ necessarily contains $IS\cap R$. We leave this observation as a labeled remark for reference in arguments that follow.

\begin{remark}[{\cite[Proposition~1.6.1]{SwansonHuneke}}]
    \label{remark extension contraction finite maps integrally closed ideals}
    Let $R\to S$ be an integral extension of Noetherian domains. If $I\subseteq R$ is an ideal of $R$, then $I\subseteq IS\cap R\subseteq \overline{IS}\cap R =\overline{I}$. In particular, if $I\subseteq R$ is integrally closed, $W_I$ the complement of the union of the associated primes of $I$, then 
    \[
    (IS_{W_I}\cap S)\cap R = IS\cap R=I.
    \]
\end{remark}

\subsection{On the Descent of the Uniform Symbolic Topology Property} The study of symbolic powers of prime ideals and finite extension $R\to S$ found in \cite{HKVfinite} builds upon an observation: If $R\to S$ is finite extension of domains, $R$ is normal, then there exists a constant $C$ so that for all prime ideals $\fp\in\Spec(R)$, $\sqrt{\fp S}^C\subseteq \overline{\fp S}$, \cite[Observation~3.1]{HKVfinite}. The following lemma is an adjustment to their observation.

\begin{lemma}
    \label{lem: linear adjutment to powers of saturations}
    Let $R\to S$ be a finite extension of excellent domains so that $R$ is normal. There exists a constant $C$ so that for all integrally closed ideals $I\subseteq R$, if $W_I$ is the complement of the union of the associated primes of $I$, $(IS_{W_I}\cap S)^{C}\subseteq \overline{IS}$.
\end{lemma}

\begin{proof}
    Let $K$ and $L$ be the fraction fields of $R$ and $S$ respectively and $[L:K]$ the degree of the finite field extension $K\to L$. We first assume that $K\to L$ is Galois. The finite extension $R_{W_I}\to S_{W_I}$ remains generically Galois. If $\sigma\in \Gal(L/K)$ then $\sigma(IS_{W_{I}}) = IS_{W_I}$ since $I\subseteq R\subseteq K$ and $S_{W_I}$ is integrally closed in $L$, therefore $\sigma(IS_{W_I}\cap S)= IS_{W_I}\cap S$. If $s\in S$ then the minimal polynomial $f_s(x)$ of $s$ over $K$ has degree at most $[L:K]$. The coefficients are integral over $R$ and therefore belong to $R$ since $R$ is integrally closed. The non-leading coefficients of $f_s(x)$ are symmetric functions of the Galois conjugates of $s$. Therefore the non-leading coefficients of $f_s(x)$ belong to 
    \[
    IS_{W_I}\cap R = (IS_{W_I}\cap R_{W_I})\cap R = IR_{W_I}\cap R = I {\footnotesize {\mbox{ by Remark~\ref{remark extension contraction finite maps integrally closed ideals}.}}}
    \]
    In particular, $s^{[L:K]}\in (s^{\deg(f_s(x))})S\subseteq IS$. Hence for all Rees valuations $\nu$ of $IS$, $\nu(s)\geq \frac{1}{[L:K]}\nu(I)$. Consequently, if $\nu$ is a Rees valuation of $IS$, then
    \[
    \nu((IS_{W_I}\cap S)^{[L:K]}) = [L:K]\nu((IS_{W_I}\cap S))\geq \nu(IS).
    \]
    Equivalently, $(IS_{W_I}\cap S)^{[L:K]}\subseteq \overline{IS}$.

    Now assume $R\to S$ is generically separable. Let $L_{\Gal}$ be the Galois closure of $K\to L$ and $T$ the integral closure of $R$ in $L_{\Gal}$. There exists $C$ so that for all integrally closed ideals $I\subseteq R$, $(IT_{W_I}\cap T)^C\subseteq \overline{IT}$. Therefore
    \[
    (IS_{W_I}\cap S)^C\subseteq (IT_{W_I}\cap T)^C\cap S\subseteq \overline{IT}\cap S = \overline{IS} {\footnotesize {\mbox{ by Remark~\ref{remark extension contraction finite maps integrally closed ideals}.}}}
    \]

    If $R\to S$ is a general finite extension of normal domains, then we can filter the field extension $K\to L$ by $K\to L_{\sep}\to L$ where $L_{\sep}$ is the separable closure of $K\to L$. If $K\to L$ are of characteristic $0$, then $L=L_{\sep}$ and there is nothing left to prove by the above. We therefore can assume $K$ and $L$ have prime characteristic $p>0$ and $e_0\in\NN$ is chosen so that $L^{p^{e_0}}\subseteq L_{\sep}$. Let $T$ be the integral closure of $R$ in $L_{\sep}$ and let $C$ be chosen so that for all integrally closed ideals $I\subseteq R$, $(IT_{W_I}\cap T)^C \subseteq \overline{IT}$. If $s\in S$ then $s^{p^{e_0}}\in L_{\sep}$ and is integral over $R$, hence belongs to $T$. Therefore 
    \[
    ((IS_{W_I}\cap S)^{C})^{[p^{e_0}]}= ((IS_{W_I}\cap S)^{[p^{e_0}]})^{C}\subseteq ((IT_{W_I}\cap T))^{C} \subseteq \overline{IT}\subseteq \overline{IS}.
    \]
    Therefore if $\nu$ is a Rees valuation of $IS$ then $\nu((IS_{W_I}\cap S)^{C})\geq \frac{1}{p^{e_0}}\nu(IS)$. In conclusion, if $\nu$ is a Rees valuation of $IS$ then
    \[
    \nu((IS_{W_I}\cap S)^{Cp^{e_0}}) = p^{e_0}\nu((IS_{W_I}\cap S)^{C})\geq \nu(IS).
    \]
    Therefore $(IS{W_I}\cap S)^{Cp^{e_0}}\subseteq \overline{IS}$.
\end{proof}

Main Theorem~\ref{Main Theorem USTP finite} is a corollary of the below Theorem~\ref{thm: USTP for normal ideals of finite extension implies USTP} and Lemma~\ref{lem: USTP for normalization and going down implies USTP}.

\begin{theorem}
    \label{thm: USTP for normal ideals of finite extension implies USTP}
    Let $R\to S$ be a finite map of excellent normal domains so that $R$ enjoys the Uniform Brian\c{c}on-Skoda Property. If $S$ enjoys the Uniform Symbolic Topology Property, then $R$ enjoys the Uniform Symbolic Topology Property.
\end{theorem}

\begin{proof}
     Let $K$ and $L$ be the fraction fields of $R$ and $S$ respectively. We first assume that $K\to L$ is separable.

    \begin{claim}
        Assume that $K\to L$ is a separable field extension.
        \begin{enumerate}
            \item\label{all associated primes galois conjugates} If $\fq\in \Ass(IS)$, then for every $\fq'\in \Spec(S)$ so that $\fq\cap R=\fq'\cap R$, $\fq'\in\Ass(IS)$.
            \item\label{all galois conjugates lying over associated primes} If $\fp\in\Ass(I)$, then every prime of $\Spec(S)$ lying over $\fp$ is an associated prime of $IS$.
        \end{enumerate}
    \end{claim}
    \begin{proof}[Proof of Claim]
        Let $L\to L_{\Gal}$ be the Galois closure of $K\to L$ and $T$ the integral closure of $R$ in $L_{\Gal}$. Let $\fq_1,\fq_2\in \Spec(S)$ lie over a common prime $\fp\in\Spec(R)$. Let $\fq_1',\fq_2'\in \Spec(T)$ be prime ideals lying over $\fq_1$ and $\fq_2$ respectively. By \cite[Ch. 5 Theorem~22]{ZariskiSamuelVol1}, there exists $\sigma\in\Gal(L_{\Gal}/K)$ so that $\sigma(\fq_1')=\fq_2'$. Hence $\sigma(\fq_1)= \sigma(\fq_1'\cap S) = \sigma(\fq_1')\cap S= \fq_2'\cap S =\fq_2$. If $\fq_1\in\Ass(IS)$ then there exists element $y\in S$ so that $\fq_1 = (IS:_S y)$. If $s\in S$, then $\sigma(s)\in \fq_2$ if and only if  $s\in \fq_1=(IS:_Sy)$ if and only if $ys\in IS$ if and only if $\sigma(s)\sigma(y)\in \sigma(IS)=IS$ if and only if $\sigma(s)\in (IS:_S\sigma(y))$. Therefore $\fq_2=(IS:_S\sigma(y))$ and $\fq_2\in \Ass(IS)$ as claimed in (\ref{all associated primes galois conjugates}).
    
         Assume that $\fp\in \Ass(I)$.  If $IS = \fq_1\cap \fq_2\cap \cdots \cap \fq_t$ is a primary decomposition of $\overline{IS}$, then by Remark~\ref{remark extension contraction finite maps integrally closed ideals},
         \[
         I=IS\cap R =(\fq_1\cap R)\cap (\fq_2\cap R)\cap \cdots \cap(\fq_t\cap R)
         \]
         is a primary decomposition of $I$. Therefore for each $\fp\in\Ass(I)$ there is some prime $\fq\in \Ass(IS)$ lying over $\fp$ so that $\fq\in\Ass(\overline{IS})$. By (\ref{all associated primes galois conjugates}), every prime lying over an associated $\fp\in \Ass(I)$ is an associated prime of $IS$ as claimed in (\ref{all galois conjugates lying over associated primes}).
    \end{proof}

    By the claim, if $R\to S$ is generically separable, if $I\subseteq R$ is an integrally closed ideal, then $I^nS_{W_I}\cap S = (IS_{W_I}\cap S)^{(n)}$. Let $D$ be a constant so that for all ideals $J\subseteq S$, $J^{(Dn)}\subseteq J^n$, and let $C_0$ be a constant as in Lemma~\ref{lem: linear adjutment to powers of saturations} so that for all integrally closed ideals $I\subseteq R$, $(IS_{W_I}\cap S)^{C_0}\subseteq \overline{IS}$. Then for all integrally closed ideals $I\subseteq R$ and $n\in\NN$,
    \[
    I^{(C_0Dn)}\subseteq (I^{C_0Dn}S_{W_I}\cap S) \cap R = (IS_{W_I}\cap S)^{(C_0Dn)}\cap R\subseteq (IS_{W_I}\cap S)^{C_0n}\cap R\subseteq \overline{IS}^n\cap R = \overline{I^n}.
    \]
    Therefore if $B$ is a Uniform Brian\c{c}on-Skoda Bound of $R$, $I\subseteq R$ an integrally closed ideal, and $n\in\NN$,
    \[
    I^{(C_0D(B+1)n)}\subseteq \overline{I^{(B+1)n}}\subseteq \overline{I^{n+B}}\subseteq I^n.
    \]
    By Lemma~\ref{lem: USTP reduces to integrally closed ideals}, $R$ enjoys the Uniform Symbolic Topology Property.

    Now suppose $K\to L$ is an arbitrary finite field extension and let $L_{\sep}$ be the separable closure of $K\to L$. Let $T$ be the integral closure of $R$ in $L_{\sep}$. For every prime $\fq\in\Spec(T)$, the ideal $\sqrt{\fq S}$ is the unique prime ideal of $S$ lying over $\fp$. Therefore if $J\subseteq T$ is an integrally closed ideal, $J^nS_{W_J}\cap S = (JS_{W_J}\cap S)^{(n)}$.
    
    By the separable argument, if $I\subseteq R$ is an integrally closed ideal, then for all $n\in\NN$, $I^nT_{W_{I}}\cap T = (IT_{W_I}\cap T)^{(n)}$. Therefore if $I\subseteq R$ is an integrally closed ideal of $R$ and $n\in\NN$,
    \[
    I^nS_{W_I}\cap S = (IT_{W_I}\cap T)^nS_{W_I}\cap S = (IS_{W_I}\cap S)^{(n)}.
    \]
    If $C_0$ is a constant as in Lemma~\ref{lem: linear adjutment to powers of saturations} so that for all integrally closed ideals $I\subseteq R$, $(IS_{W_{I}}\cap S)^{C_0}\subseteq \overline{IS}$ and $D$ is a constant so that for all $J\subseteq S$ and $n\in\NN$, $J^{(Dn)}\subseteq J^n$, then for all integrally closed ideals $I\subseteq R$ and $n\in\NN$,
    \[
    I^{(C_0Dn)}\subseteq (I^{C_0Dn}S_{W_I}\cap S) \cap R = (IS_{W_I}\cap S)^{(C_0Dn)}\cap R\subseteq (IS_{W_I}\cap S)^{C_0n}\cap R\subseteq \overline{IS}^n\cap R = \overline{I^n}.
    \]
    As in the separable case, if $B$ is a Uniform Brian\c{c}on-Skoda Bound of $R$, then for all integrally closed ideals $I\subseteq R$ and $n\in\NN$, $I^{(C_0D(B+1)n)}\subseteq I^n$.  By Lemma~\ref{lem: USTP reduces to integrally closed ideals}, $R$ enjoys the Uniform Symbolic Topology Property.
\end{proof}

If $k$ is a field of characteristic $0$, $G$ a subgroup of $S_n$, $G$ acts on $k[x_1,x_2,\ldots,x_n]$ by permuting the variables, $k[x_1,x_2,\ldots,x_n]^G$ the ring of invariants, then the ring extension $k[x_1,x_2,\ldots,x_n]^G\to k[x_1,x_2,\ldots,x_n]$ is a finite extension of normal domains. An application of \cite[Corollary~3.5]{HKVfinite} is that the ring of invariants $k[x_1,x_2,\ldots,x_n]^G$ enjoys the Uniform Symbolic Topology Property for prime ideals. An immediate corollary then of Theorem~\ref{thm: USTP for normal ideals of finite extension implies USTP} is a generalization of \cite[Corollary~3.5]{HKVfinite} to conclude the ring of invariants enjoys the Uniform Symbolic Topology Property for all ideals, not just prime ideals.

\begin{corollary}
    Let $S$ be a regular ring containing a field of characteristic $0$, $G$ a finite group acting on $S$, and $R^G$ the ring of invariants. If $R^G$ enjoys the Uniform Brian\c{c}on-Skoda Property, e.g. $S$ is complete or essentially of finite type over a field, then $R^G$ enjoys the Uniform Symbolic Topology Property.
\end{corollary}

\subsection{Proof of Main Theorem~\ref{Main theorem USTP in SFR rings}}

Theorem~\ref{thm main theorem} below is a mild generalization of the statement found in Main Theorem~\ref{Main theorem USTP in SFR rings}. Theorem~\ref{thm main theorem} only requires the non-normal locus of the base ring $R$ to only consist of isolated points. To this end we begin with a lemma.

\begin{lemma}
\label{lem: USTP for normalization and going down implies USTP}
    Let $R$ be an excellent domain that enjoys the Uniform Brian\c{c}on-Skoda Property and the non-normal locus of $\Spec(R)$ only consists of isolated points. If $\overline{R}$ enjoys the Uniform Symbolic Topology Property then $R$ enjoys the Uniform Symbolic Topology Property.
\end{lemma}

\begin{proof}
    The normalization map $R\to \overline{R}$ is finite because $R$ is excellent. The conductor ideal $\fc:=\Ann_R(\overline{R}/R)$ is an intersection of finitely many ideals primary to a maximal ideal in the non-normal locus of $\Spec(R)$. 

    \begin{claim}
        \label{claim uniform C to raise saturated ideal}
        There exists a constant $C$ so that for all integrally closed ideals $I\subseteq R$, if $W_I$ the complement of the union of the associated primes of $I$, then $(I\overline{R}_{W_I}\cap \overline{R})^C\subseteq I$.
    \end{claim}

    \begin{proof}[Proof of Claim]
        It suffices to find a constant $C$ so that for each maximal ideal $\fm\in\Spec(R)$, $(I(\overline{R}_{\fm})_{W_I}\cap \overline{R}_{\fm})^C\subseteq \overline{IR}_{\fm}$. The constant $C=1$ works if $\fm$ belongs to the normal locus of $\Spec(R)$. If $\fm$ is one of the finitely many maximal ideals of the non-normal locus of $\Spec(R)$, then $\fc R_{\fm}$ is primary to the maximal ideal $\fm R_{\fm}$. So there exists $C_{\fm}$ so that $\fm^{C_{\fm}}\overline{R}_{\fm}\subseteq \fc \overline{R}_{\fm}$. Moreover, 
        \[
        \fc (I(\overline{R}_{\fm})_{W_I}\cap \overline{R}_{\fm})\subseteq (I(\overline{R}_{\fm})_{W_I}\cap \overline{R}_{\fm}) \cap R_{\fm} = IR_{\fm} {\footnotesize{ \mbox{ by Remark~\ref{remark extension contraction finite maps integrally closed ideals}.}}}
        \]
        Therefore 
        \begin{align*}
            (I(\overline{R}_{\fm})_{W_I}\cap \overline{R}_{\fm})^{C_{\fm}+1}&\subseteq \fm^{C_{\fm}}(I(\overline{R}_{\fm})_{W_I}\cap \overline{R}_{\fm})\\
            &\subseteq \fc(I(\overline{R}_{\fm})_{W_I}\cap \overline{R}_{\fm})\\
            &\subseteq IR_{\fm}.
        \end{align*}
        The constant $C = \max\{C_{\fm}+1\mid R_{\fm} \mbox{ is non-normal}\}$ has the claimed property.
    \end{proof}
    
    Let $C$ be as in Claim~\ref{claim uniform C to raise saturated ideal} and $D$ be a constant so that for all ideals $J\subseteq \overline{R}$ and $n\in\NN$, $J^{(Dn)}\subseteq J^n$. Then for all integrally closed ideals $I\subseteq R$ and $n\in\NN$,
    \begin{align*}
        I^{(CDn)} &\subseteq I^{CDn}\overline{R}_{W_I}\cap \overline{R}\\
        & = (I\overline{R}_{W_I}\cap \overline{R})^{(CDn)}\\
        & \subseteq (I\overline{R}_{W_I}\cap \overline{R})^{Cn}\\
        &\subseteq I^n.
    \end{align*}
    By Lemma~\ref{lem: USTP reduces to integrally closed ideals}, $R$ enjoys the Uniform Symbolic Topology Property.
\end{proof}

We now have the necessary ingredients to complete the proof of Main Theorem~\ref{Main theorem USTP in SFR rings}.

\begin{theorem}
    \label{thm main theorem}
    Let $R$ be a domain of prime characteristic $p>0$ that is either $F$-finite or essentially of finite type over an excellent local ring. Assume that the non-normal locus of $\Spec(R)$ consists only of isolated points. If there exists a finite extension $R\to S$ so that $S_\fp$ is strongly $F$-regular for all non-maximal prime ideals $\fp\in\Spec(S)$, then $R$ enjoys the Uniform Symbolic Topology Property.
\end{theorem}

\begin{proof}
    The induced maps of normalizations $\overline{R}\to \overline{S}$ is finite. Normalization commutes with localization, strong $F$-regularity is a local condition, and strongly $F$-regular rings are normal. Therefore $\overline{S}$ still enjoys the property that $\overline{S}_{\fp}$ is strongly $F$-regular for all non-maximal prime ideals $\fp\in\Spec(\overline{S})$. 

    The rings $R$ and $\overline{R}$ enjoy the Uniform Brian\c{c}on-Skoda Property by \cite[Theorem~4.13]{HunekeUniformBounds}. The ring $\overline{S}$ enjoys the Uniform Symbolic Topology Property by Corollary~\ref{corollary USTP non-F-finite rings}. Therefore $\overline{R}$ enjoys the Uniform Symbolic Topology Property by Theorem~\ref{thm: USTP for normal ideals of finite extension implies USTP}. The ring $R$ enjoys the Uniform Symbolic Topology Property by Lemma~\ref{lem: USTP for normalization and going down implies USTP}.
\end{proof}

\begin{corollary}
    \label{corollary USTP for purely inseparable extension}
    Let $A$ be a domain of prime characteristic $p>0$ that is either $F$-finite or essentially of finite type over an excellent local ring. Assume that the non-strongly $F$-regular locus of $\Spec(R)$ only consists of isolated points. Let $A\to R$ be a finite and purely inseparable extension of domains so that the non-normal locus of $\Spec(R)$ only consists of isolated points. Then $R$ enjoys the Uniform Symbolic Topology Property.
\end{corollary}

\begin{proof}
    Because $A\to R$ is finite and purely inseparable, there exists $e\in\NN$ so that $R\subseteq A^{1/p^e}$ is a finite extension. The ring $A^{1/p^e}$ is abstractly isomorphic to $A$. Therefore $R$ enjoys the Uniform Symbolic Topology Property by Theorem~\ref{thm main theorem}.
\end{proof}

%% file: main.bib
@book {ZariskiSamuelVol1,
    AUTHOR = {Zariski, Oscar and Samuel, Pierre},
     TITLE = {Commutative algebra, {V}olume {I}},
    SERIES = {The University Series in Higher Mathematics},
      NOTE = {With the cooperation of I. S. Cohen},
 PUBLISHER = {D. Van Nostrand Co., Inc., Princeton, NJ},
      YEAR = {1958},
     PAGES = {xi+329},
   MRCLASS = {09.3X},
  MRNUMBER = {90581},
MRREVIEWER = {Melvin\ Henriksen},
}

@article {Hsiao,
    AUTHOR = {Hsiao, Jen-Chieh},
     TITLE = {A counterexample for subadditivity of multiplier ideals on
              toric varieties},
   JOURNAL = {Comm. Algebra},
  FJOURNAL = {Communications in Algebra},
    VOLUME = {40},
      YEAR = {2012},
    NUMBER = {5},
     PAGES = {1618--1624},
      ISSN = {0092-7872,1532-4125},
   MRCLASS = {14F18 (14M25)},
  MRNUMBER = {2924471},
MRREVIEWER = {Shin-Yao\ Jow},
       DOI = {10.1080/00927872.2011.552084},
       URL = {https://doi.org/10.1080/00927872.2011.552084},
}

@article {TWSubadditivityFails,
    AUTHOR = {Takagi, Shunsuke and Watanabe, Kei-ichi},
     TITLE = {When does the subadditivity theorem for multiplier ideals
              hold?},
   JOURNAL = {Trans. Amer. Math. Soc.},
  FJOURNAL = {Transactions of the American Mathematical Society},
    VOLUME = {356},
      YEAR = {2004},
    NUMBER = {10},
     PAGES = {3951--3961},
      ISSN = {0002-9947,1088-6850},
   MRCLASS = {13B22 (14J17)},
  MRNUMBER = {2058513},
MRREVIEWER = {Karen\ E.\ Smith},
       DOI = {10.1090/S0002-9947-04-03436-1},
       URL = {https://doi.org/10.1090/S0002-9947-04-03436-1},
}

@article {SchwedeCenters,
    AUTHOR = {Schwede, Karl},
     TITLE = {Centers of {$F$}-purity},
   JOURNAL = {Math. Z.},
  FJOURNAL = {Mathematische Zeitschrift},
    VOLUME = {265},
      YEAR = {2010},
    NUMBER = {3},
     PAGES = {687--714},
      ISSN = {0025-5874,1432-1823},
   MRCLASS = {13A35 (14B05 14F18)},
  MRNUMBER = {2644316},
MRREVIEWER = {Doru\ \c Stef\u anescu},
       DOI = {10.1007/s00209-009-0536-5},
       URL = {https://doi.org/10.1007/s00209-009-0536-5},
}

@article {HaraTestIdeals,
    AUTHOR = {Hara, Nobuo},
     TITLE = {A characteristic {$p$} analog of multiplier ideals and
              applications},
   JOURNAL = {Comm. Algebra},
  FJOURNAL = {Communications in Algebra},
    VOLUME = {33},
      YEAR = {2005},
    NUMBER = {10},
     PAGES = {3375--3388},
      ISSN = {0092-7872,1532-4125},
   MRCLASS = {13A35 (14B05)},
  MRNUMBER = {2175438},
MRREVIEWER = {Adela\ N.\ Vraciu},
       DOI = {10.1080/AGB-200060022},
       URL = {https://doi.org/10.1080/AGB-200060022},
}

@article {HochsterYaoTrans,
    AUTHOR = {Hochster, Melvin and Yao, Yongwei},
     TITLE = {Frobenius splitting, strong {F}-regularity, and small
              {C}ohen-{M}acaulay modules},
   JOURNAL = {Trans. Amer. Math. Soc.},
  FJOURNAL = {Transactions of the American Mathematical Society},
    VOLUME = {376},
      YEAR = {2023},
    NUMBER = {9},
     PAGES = {6729--6765},
      ISSN = {0002-9947,1088-6850},
   MRCLASS = {13A35 (13C14 13D45 13F40 13H05 13H10)},
  MRNUMBER = {4630790},
MRREVIEWER = {Mitra\ Koley},
       DOI = {10.1090/tran/8964},
       URL = {https://doi.org/10.1090/tran/8964},
}

@incollection {BlickleMustataSmith,
    AUTHOR = {Blickle, Manuel and Musta\c{t}\v{a}, Mircea and Smith, Karen E.},
     TITLE = {Discreteness and rationality of {$F$}-thresholds},
      NOTE = {Special volume in honor of Melvin Hochster},
   JOURNAL = {Michigan Math. J.},
  FJOURNAL = {Michigan Mathematical Journal},
    VOLUME = {57},
      YEAR = {2008},
     PAGES = {43--61},
      ISSN = {0026-2285,1945-2365},
   MRCLASS = {13A35 (14F18)},
  MRNUMBER = {2492440},
MRREVIEWER = {Tomasz\ Szemberg},
       DOI = {10.1307/mmj/1220879396},
       URL = {https://doi.org/10.1307/mmj/1220879396},
}

@article {MaSchwedeSingularities,
    AUTHOR = {Ma, Linquan and Schwede, Karl},
     TITLE = {Singularities in mixed characteristic via perfectoid big
              {C}ohen-{M}acaulay algebras},
   JOURNAL = {Duke Math. J.},
  FJOURNAL = {Duke Mathematical Journal},
    VOLUME = {170},
      YEAR = {2021},
    NUMBER = {13},
     PAGES = {2815--2890},
      ISSN = {0012-7094,1547-7398},
   MRCLASS = {14G45 (13A35 14B05 14D10 14F18)},
  MRNUMBER = {4312190},
MRREVIEWER = {Ana\ Bravo},
       DOI = {10.1215/00127094-2020-0082},
       URL = {https://doi.org/10.1215/00127094-2020-0082},
}

@article {SinghSwanson,
    AUTHOR = {Singh, Anurag K. and Swanson, Irena},
     TITLE = {Associated primes of local cohomology modules and of
              {F}robenius powers},
   JOURNAL = {Int. Math. Res. Not.},
  FJOURNAL = {International Mathematics Research Notices},
      YEAR = {2004},
    NUMBER = {33},
     PAGES = {1703--1733},
      ISSN = {1073-7928,1687-0247},
   MRCLASS = {13D45},
  MRNUMBER = {2058025},
MRREVIEWER = {Manuel\ Blickle},
       DOI = {10.1155/S1073792804133424},
       URL = {https://doi.org/10.1155/S1073792804133424},
}

@article {Katzman,
    AUTHOR = {Katzman, Mordechai},
     TITLE = {An example of an infinite set of associated primes of a local
              cohomology module},
   JOURNAL = {J. Algebra},
  FJOURNAL = {Journal of Algebra},
    VOLUME = {252},
      YEAR = {2002},
    NUMBER = {1},
     PAGES = {161--166},
      ISSN = {0021-8693,1090-266X},
   MRCLASS = {13D45},
  MRNUMBER = {1922391},
MRREVIEWER = {I-Chiau\ Huang},
       DOI = {10.1016/S0021-8693(02)00032-7},
       URL = {https://doi.org/10.1016/S0021-8693(02)00032-7},
}

@article {TuckerFsigExists,
    AUTHOR = {Tucker, Kevin},
     TITLE = {{$F$}-signature exists},
   JOURNAL = {Invent. Math.},
  FJOURNAL = {Inventiones Mathematicae},
    VOLUME = {190},
      YEAR = {2012},
    NUMBER = {3},
     PAGES = {743--765},
      ISSN = {0020-9910,1432-1297},
   MRCLASS = {13A35 (13D40 14B05)},
  MRNUMBER = {2995185},
MRREVIEWER = {Alberto\ F.\ Boix},
       DOI = {10.1007/s00222-012-0389-0},
       URL = {https://doi.org/10.1007/s00222-012-0389-0},
}

@article {SchwedeSmith,
    AUTHOR = {Schwede, Karl and Smith, Karen E.},
     TITLE = {Globally {$F$}-regular and log {F}ano varieties},
   JOURNAL = {Adv. Math.},
  FJOURNAL = {Advances in Mathematics},
    VOLUME = {224},
      YEAR = {2010},
    NUMBER = {3},
     PAGES = {863--894},
      ISSN = {0001-8708,1090-2082},
   MRCLASS = {14J45 (13A35 14B05)},
  MRNUMBER = {2628797},
MRREVIEWER = {Yukihide\ Takayama},
       DOI = {10.1016/j.aim.2009.12.020},
       URL = {https://doi.org/10.1016/j.aim.2009.12.020},
}

@misc{PolstraZariskiNagata,
      title={Zariski-Nagata Theorems for Singularities and the Uniform Izumi-Rees Property}, 
      author={Thomas Polstra},
      year={2024},
      eprint={2406.00759},
      archivePrefix={arXiv},
      primaryClass={id='math.AC' full_name='Commutative Algebra' is_active=True alt_name=None in_archive='math' is_general=False description='Commutative rings, modules, ideals, homological algebra, computational aspects, invariant theory, connections to algebraic geometry and combinatorics'}
}

@misc{CaminataShidelerTuckerZerman,
      title={F-signature functions of diagonal hypersurfaces}, 
      author={Alessio Caminata and Samuel Shideler and Kevin Tucker and Francesco Zerman},
      year={2024},
      eprint={2403.12863},
      archivePrefix={arXiv},
      primaryClass={math.AC},
      url={https://arxiv.org/abs/2403.12863}, 
}

@article {MehtaSrinivas,
    AUTHOR = {Mehta, V. B. and Srinivas, V.},
     TITLE = {A characterization of rational singularities},
   JOURNAL = {Asian J. Math.},
  FJOURNAL = {Asian Journal of Mathematics},
    VOLUME = {1},
      YEAR = {1997},
    NUMBER = {2},
     PAGES = {249--271},
      ISSN = {1093-6106,1945-0036},
   MRCLASS = {13A99 (14B05)},
  MRNUMBER = {1491985},
MRREVIEWER = {Karen\ E.\ Smith},
       DOI = {10.4310/AJM.1997.v1.n2.a4},
       URL = {https://doi.org/10.4310/AJM.1997.v1.n2.a4},
}

@article {SmithRational,
    AUTHOR = {Smith, Karen E.},
     TITLE = {{$F$}-rational rings have rational singularities},
   JOURNAL = {Amer. J. Math.},
  FJOURNAL = {American Journal of Mathematics},
    VOLUME = {119},
      YEAR = {1997},
    NUMBER = {1},
     PAGES = {159--180},
      ISSN = {0002-9327,1080-6377},
   MRCLASS = {13A35 (13D45 13F40 14B05)},
  MRNUMBER = {1428062},
MRREVIEWER = {Ian\ M.\ Aberbach},
       URL =
              {http://muse.jhu.edu/journals/american_journal_of_mathematics/v119/119.1smith.pdf},
}

@article {HaraRational,
    AUTHOR = {Hara, Nobuo},
     TITLE = {A characterization of rational singularities in terms of
              injectivity of {F}robenius maps},
   JOURNAL = {Amer. J. Math.},
  FJOURNAL = {American Journal of Mathematics},
    VOLUME = {120},
      YEAR = {1998},
    NUMBER = {5},
     PAGES = {981--996},
      ISSN = {0002-9327,1080-6377},
   MRCLASS = {13A99 (14B05 14E05)},
  MRNUMBER = {1646049},
MRREVIEWER = {Karen\ E.\ Smith},
       URL =
              {http://muse.jhu.edu/journals/american_journal_of_mathematics/v120/120.5hara.pdf},
}

@article {HaraWatanabe,
    AUTHOR = {Hara, Nobuo and Watanabe, Kei-Ichi},
     TITLE = {F-regular and {F}-pure rings vs. log terminal and log
              canonical singularities},
   JOURNAL = {J. Algebraic Geom.},
  FJOURNAL = {Journal of Algebraic Geometry},
    VOLUME = {11},
      YEAR = {2002},
    NUMBER = {2},
     PAGES = {363--392},
      ISSN = {1056-3911,1534-7486},
   MRCLASS = {13A35 (14B05 14J17)},
  MRNUMBER = {1874118},
MRREVIEWER = {Karen\ E.\ Smith},
       DOI = {10.1090/S1056-3911-01-00306-X},
       URL = {https://doi.org/10.1090/S1056-3911-01-00306-X},
}

@incollection {HHmem,
    AUTHOR = {Hochster, Melvin and Huneke, Craig},
     TITLE = {Tight closure and strong {$F$}-regularity},
      NOTE = {Colloque en l'honneur de Pierre Samuel (Orsay, 1987)},
   JOURNAL = {M\'em. Soc. Math. France (N.S.)},
  FJOURNAL = {M\'emoires de la Soci\'et\'e{} Math\'ematique de France.
              Nouvelle S\'erie},
    NUMBER = {38},
      YEAR = {1989},
     PAGES = {119--133},
      ISSN = {0037-9484},
   MRCLASS = {13H10 (13A50 13C14)},
  MRNUMBER = {1044348},
MRREVIEWER = {W.\ V.\ Vasconcelos},
}

@book{HochsterFoundations,
  author    = {Hochster, Melvin},
  title     = {Foundations of Tight Closure},
  note       = {https://dept.math.lsa.umich.edu /~hochster/615W22/Fnd.Tcl.Lx.pdf, Lecture notes from a course taught at the University of Michigan, Winter 2022.},
}

@article {SmithVanDenBergh,
    AUTHOR = {Smith, Karen E. and Van den Bergh, Michel},
     TITLE = {Simplicity of rings of differential operators in prime
              characteristic},
   JOURNAL = {Proc. London Math. Soc. (3)},
  FJOURNAL = {Proceedings of the London Mathematical Society. Third Series},
    VOLUME = {75},
      YEAR = {1997},
    NUMBER = {1},
     PAGES = {32--62},
      ISSN = {0024-6115,1460-244X},
   MRCLASS = {16S32},
  MRNUMBER = {1444312},
MRREVIEWER = {Ian\ M.\ Musson},
       DOI = {10.1112/S0024611597000257},
       URL = {https://doi.org/10.1112/S0024611597000257},
}

@article {HunekeLeuschke,
    AUTHOR = {Huneke, Craig and Leuschke, Graham J.},
     TITLE = {Two theorems about maximal {C}ohen-{M}acaulay modules},
   JOURNAL = {Math. Ann.},
  FJOURNAL = {Mathematische Annalen},
    VOLUME = {324},
      YEAR = {2002},
    NUMBER = {2},
     PAGES = {391--404},
      ISSN = {0025-5831,1432-1807},
   MRCLASS = {13C14 (13A35 13D07 13H10)},
  MRNUMBER = {1933863},
MRREVIEWER = {\=I.\ \=I.\ Burban},
       DOI = {10.1007/s00208-002-0343-3},
       URL = {https://doi.org/10.1007/s00208-002-0343-3},
}

@article {HochsterPurity,
    AUTHOR = {Hochster, Melvin},
     TITLE = {Cyclic purity versus purity in excellent {N}oetherian rings},
   JOURNAL = {Trans. Amer. Math. Soc.},
  FJOURNAL = {Transactions of the American Mathematical Society},
    VOLUME = {231},
      YEAR = {1977},
    NUMBER = {2},
     PAGES = {463--488},
      ISSN = {0002-9947,1088-6850},
   MRCLASS = {13D99},
  MRNUMBER = {463152},
MRREVIEWER = {Toma\ Albu},
       DOI = {10.2307/1997914},
       URL = {https://doi.org/10.2307/1997914},
}

@article {HHJAMS,
    AUTHOR = {Hochster, Melvin and Huneke, Craig},
     TITLE = {Tight closure, invariant theory, and the {B}rian\c con-{S}koda
              theorem.},
   JOURNAL = {J. Amer. Math. Soc.},
  FJOURNAL = {Journal of the American Mathematical Society},
    VOLUME = {3},
      YEAR = {1990},
    NUMBER = {1},
     PAGES = {31--116},
      ISSN = {0894-0347,1088-6834},
   MRCLASS = {13C05 (13A15 13A50 13B99 13D02)},
  MRNUMBER = {1017784},
MRREVIEWER = {Luchezar\ L.\ Avramov},
       DOI = {10.2307/1990984},
       URL = {https://doi.org/10.2307/1990984},
}

@article {KunzPrimeRegular,
    AUTHOR = {Kunz, Ernst},
     TITLE = {Characterizations of regular local rings of characteristic
              {$p$}},
   JOURNAL = {Amer. J. Math.},
  FJOURNAL = {American Journal of Mathematics},
    VOLUME = {91},
      YEAR = {1969},
     PAGES = {772--784},
      ISSN = {0002-9327,1080-6377},
   MRCLASS = {13.95},
  MRNUMBER = {252389},
MRREVIEWER = {M.\ Nagata},
       DOI = {10.2307/2373351},
       URL = {https://doi.org/10.2307/2373351},
}

@article {HaraYoshida,
    AUTHOR = {Hara, Nobuo and Yoshida, Ken-Ichi},
     TITLE = {A generalization of tight closure and multiplier ideals},
   JOURNAL = {Trans. Amer. Math. Soc.},
  FJOURNAL = {Transactions of the American Mathematical Society},
    VOLUME = {355},
      YEAR = {2003},
    NUMBER = {8},
     PAGES = {3143--3174},
      ISSN = {0002-9947,1088-6850},
   MRCLASS = {13A35},
  MRNUMBER = {1974679},
MRREVIEWER = {Irena\ Swanson},
       DOI = {10.1090/S0002-9947-03-03285-9},
       URL = {https://doi.org/10.1090/S0002-9947-03-03285-9},
}

@article {HKAbelian,
    AUTHOR = {Huneke, Craig and Katz, Daniel},
     TITLE = {Uniform symbolic topologies in abelian extensions},
   JOURNAL = {Trans. Amer. Math. Soc.},
  FJOURNAL = {Transactions of the American Mathematical Society},
    VOLUME = {372},
      YEAR = {2019},
    NUMBER = {3},
     PAGES = {1735--1750},
      ISSN = {0002-9947,1088-6850},
   MRCLASS = {13A02 (13F20 13H15)},
  MRNUMBER = {3976575},
MRREVIEWER = {Aryampilly\ V.\ Jayanthan},
       DOI = {10.1090/tran/7623},
       URL = {https://doi.org/10.1090/tran/7623},
}

@article {PolstraSmirnovEquimultiplicity,
    AUTHOR = {Polstra, Thomas and Smirnov, Ilya},
     TITLE = {Equimultiplicity theory of strongly {$F$}-regular rings},
   JOURNAL = {Michigan Math. J.},
  FJOURNAL = {Michigan Mathematical Journal},
    VOLUME = {70},
      YEAR = {2021},
    NUMBER = {4},
     PAGES = {837--856},
      ISSN = {0026-2285,1945-2365},
   MRCLASS = {13A35},
  MRNUMBER = {4332680},
MRREVIEWER = {Jack\ Jeffries},
       DOI = {10.1307/mmj/1600913073},
       URL = {https://doi.org/10.1307/mmj/1600913073},
}

@article {HKHypersurface,
    AUTHOR = {Huneke, Craig and Katz, Daniel},
     TITLE = {Uniform {S}ymbolic {T}opologies and {H}ypersurfaces},
   JOURNAL = {Acta Math. Vietnam.},
  FJOURNAL = {Acta Mathematica Vietnamica},
    VOLUME = {49},
      YEAR = {2024},
    NUMBER = {1},
     PAGES = {99--113},
      ISSN = {0251-4184,2315-4144},
   MRCLASS = {99-06},
  MRNUMBER = {4754259},
       DOI = {10.1007/s40306-024-00526-8},
       URL = {https://doi.org/10.1007/s40306-024-00526-8},
}

@article {Walker1,
    AUTHOR = {Walker, Robert M.},
     TITLE = {Uniform symbolic topologies in normal toric rings},
   JOURNAL = {J. Algebra},
  FJOURNAL = {Journal of Algebra},
    VOLUME = {511},
      YEAR = {2018},
     PAGES = {292--298},
      ISSN = {0021-8693,1090-266X},
   MRCLASS = {13H10 (14C20 14M25)},
  MRNUMBER = {3834775},
MRREVIEWER = {Adam\ L.\ Van Tuyl},
       DOI = {10.1016/j.jalgebra.2018.05.038},
       URL = {https://doi.org/10.1016/j.jalgebra.2018.05.038},
}

@article {Walker3,
    AUTHOR = {Walker, Robert M.},
     TITLE = {Uniform symbolic topologies via multinomial expansions},
   JOURNAL = {Proc. Amer. Math. Soc.},
  FJOURNAL = {Proceedings of the American Mathematical Society},
    VOLUME = {146},
      YEAR = {2018},
    NUMBER = {9},
     PAGES = {3735--3746},
      ISSN = {0002-9939,1088-6826},
   MRCLASS = {13A15 (13H10 14C20 14M25)},
  MRNUMBER = {3825829},
MRREVIEWER = {Karl\ Schwede},
       DOI = {10.1090/proc/14073},
       URL = {https://doi.org/10.1090/proc/14073},
}

@article {Walker2,
    AUTHOR = {Walker, Robert M.},
     TITLE = {Rational singularities and uniform symbolic topologies},
   JOURNAL = {Illinois J. Math.},
  FJOURNAL = {Illinois Journal of Mathematics},
    VOLUME = {60},
      YEAR = {2016},
    NUMBER = {2},
     PAGES = {541--550},
      ISSN = {0019-2082,1945-6581},
   MRCLASS = {13H10 (14B05 14M25)},
  MRNUMBER = {3680547},
MRREVIEWER = {Reza\ Naghipour},
       URL = {http://projecteuclid.org/euclid.ijm/1499760021},
}

@article {DiagonallyFRegular,
    AUTHOR = {Carvajal-Rojas, Javier and Smolkin, Daniel},
     TITLE = {The uniform symbolic topology property for diagonally
              {$F$}-regular algebras},
   JOURNAL = {J. Algebra},
  FJOURNAL = {Journal of Algebra},
    VOLUME = {548},
      YEAR = {2020},
     PAGES = {25--52},
      ISSN = {0021-8693,1090-266X},
   MRCLASS = {13A15 (13A35 14B05)},
  MRNUMBER = {4044703},
MRREVIEWER = {Dinh\ Thanh\ Trung},
       DOI = {10.1016/j.jalgebra.2019.11.017},
       URL = {https://doi.org/10.1016/j.jalgebra.2019.11.017},
}

@article {DeStefaniGrifoJeffries,
    AUTHOR = {De Stefani, Alessandro and Grifo, Elo\'{\i}sa and Jeffries,
              Jack},
     TITLE = {A {Z}ariski-{N}agata theorem for smooth {$\Bbb{Z}$}-algebras},
   JOURNAL = {J. Reine Angew. Math.},
  FJOURNAL = {Journal f\"{u}r die Reine und Angewandte Mathematik. [Crelle's
              Journal]},
    VOLUME = {761},
      YEAR = {2020},
     PAGES = {123--140},
      ISSN = {0075-4102,1435-5345},
   MRCLASS = {13A35 (14G45)},
  MRNUMBER = {4080246},
MRREVIEWER = {Ana\ Bravo},
       DOI = {10.1515/crelle-2018-0012},
       URL = {https://doi.org/10.1515/crelle-2018-0012},
}

@article {MurayamaGamma,
    AUTHOR = {Murayama, Takumi},
     TITLE = {The gamma construction and asymptotic invariants of line
              bundles over arbitrary fields},
   JOURNAL = {Nagoya Math. J.},
  FJOURNAL = {Nagoya Mathematical Journal},
    VOLUME = {242},
      YEAR = {2021},
     PAGES = {165--207},
      ISSN = {0027-7630,2152-6842},
   MRCLASS = {14C20 (13A35 14B05 14F17)},
  MRNUMBER = {4250735},
MRREVIEWER = {Kevin\ Tucker},
       DOI = {10.1017/nmj.2019.27},
       URL = {https://doi.org/10.1017/nmj.2019.27},
}

@article {McAdamRatliffSymbolic,
    AUTHOR = {McAdam, S. and Ratliff, Jr., L. J.},
     TITLE = {Note on symbolic powers and going down},
   JOURNAL = {Proc. Amer. Math. Soc.},
  FJOURNAL = {Proceedings of the American Mathematical Society},
    VOLUME = {98},
      YEAR = {1986},
    NUMBER = {2},
     PAGES = {199--204},
      ISSN = {0002-9939,1088-6826},
   MRCLASS = {13C15 (13A17)},
  MRNUMBER = {854018},
MRREVIEWER = {P.\ Schenzel},
       DOI = {10.2307/2045683},
       URL = {https://doi.org/10.2307/2045683},
}

@article {Schenzel,
    AUTHOR = {Schenzel, Peter},
     TITLE = {Symbolic powers of prime ideals and their topology},
   JOURNAL = {Proc. Amer. Math. Soc.},
  FJOURNAL = {Proceedings of the American Mathematical Society},
    VOLUME = {93},
      YEAR = {1985},
    NUMBER = {1},
     PAGES = {15--20},
      ISSN = {0002-9939,1088-6826},
   MRCLASS = {13C15 (13A17 13H10 13J99)},
  MRNUMBER = {766518},
MRREVIEWER = {Markus\ Brodmann},
       DOI = {10.2307/2044544},
       URL = {https://doi.org/10.2307/2044544},
}

@article {HKVfinite,
    AUTHOR = {Huneke, Craig and Katz, Daniel and Validashti, Javid},
     TITLE = {Uniform symbolic topologies and finite extensions},
   JOURNAL = {J. Pure Appl. Algebra},
  FJOURNAL = {Journal of Pure and Applied Algebra},
    VOLUME = {219},
      YEAR = {2015},
    NUMBER = {3},
     PAGES = {543--550},
      ISSN = {0022-4049,1873-1376},
   MRCLASS = {13A15 (13F20 13H15)},
  MRNUMBER = {3279373},
MRREVIEWER = {Elena\ Guardo},
       DOI = {10.1016/j.jpaa.2014.05.012},
       URL = {https://doi.org/10.1016/j.jpaa.2014.05.012},
}

@article {DSGJ,
    AUTHOR = {De Stefani, Alessandro and Grifo, Elo\'{\i}sa and Jeffries,
              Jack},
     TITLE = {A uniform {C}hevalley theorem for direct summands of
              polynomial rings in mixed characteristic},
   JOURNAL = {Math. Z.},
  FJOURNAL = {Mathematische Zeitschrift},
    VOLUME = {301},
      YEAR = {2022},
    NUMBER = {4},
     PAGES = {4141--4151},
      ISSN = {0025-5874,1432-1823},
   MRCLASS = {13A35 (13A02)},
  MRNUMBER = {4449742},
MRREVIEWER = {Dang\ Hop\ Nguyen},
       DOI = {10.1007/s00209-022-03035-2},
       URL = {https://doi.org/10.1007/s00209-022-03035-2},
}

@article {KunzExcellent,
    AUTHOR = {Kunz, Ernst},
     TITLE = {On {N}oetherian rings of characteristic {$p$}},
   JOURNAL = {Amer. J. Math.},
  FJOURNAL = {American Journal of Mathematics},
    VOLUME = {98},
      YEAR = {1976},
    NUMBER = {4},
     PAGES = {999--1013},
      ISSN = {0002-9327,1080-6377},
   MRCLASS = {13E05 (14E15)},
  MRNUMBER = {432625},
MRREVIEWER = {Hamet\ Seydi},
       DOI = {10.2307/2374038},
       URL = {https://doi.org/10.2307/2374038},
}

@article {AberbachLeuschke,
    AUTHOR = {Aberbach, Ian M. and Leuschke, Graham J.},
     TITLE = {The {$F$}-signature and strong {$F$}-regularity},
   JOURNAL = {Math. Res. Lett.},
  FJOURNAL = {Mathematical Research Letters},
    VOLUME = {10},
      YEAR = {2003},
    NUMBER = {1},
     PAGES = {51--56},
      ISSN = {1073-2780},
   MRCLASS = {13A35},
  MRNUMBER = {1960123},
MRREVIEWER = {Karen\ E.\ Smith},
       DOI = {10.4310/MRL.2003.v10.n1.a6},
       URL = {https://doi.org/10.4310/MRL.2003.v10.n1.a6},
}

@article {AberbachEnescu,
    AUTHOR = {Aberbach, Ian M. and Enescu, Florian},
     TITLE = {The structure of {F}-pure rings},
   JOURNAL = {Math. Z.},
  FJOURNAL = {Mathematische Zeitschrift},
    VOLUME = {250},
      YEAR = {2005},
    NUMBER = {4},
     PAGES = {791--806},
      ISSN = {0025-5874,1432-1823},
   MRCLASS = {13A35},
  MRNUMBER = {2180375},
MRREVIEWER = {Nobuo\ Hara},
       DOI = {10.1007/s00209-005-0776-y},
       URL = {https://doi.org/10.1007/s00209-005-0776-y},
}

@article {YaoObservations,
    AUTHOR = {Yao, Yongwei},
     TITLE = {Observations on the {$F$}-signature of local rings of
              characteristic {$p$}},
   JOURNAL = {J. Algebra},
  FJOURNAL = {Journal of Algebra},
    VOLUME = {299},
      YEAR = {2006},
    NUMBER = {1},
     PAGES = {198--218},
      ISSN = {0021-8693,1090-266X},
   MRCLASS = {13A35 (13H10)},
  MRNUMBER = {2225772},
MRREVIEWER = {Ian\ M.\ Aberbach},
       DOI = {10.1016/j.jalgebra.2005.08.013},
       URL = {https://doi.org/10.1016/j.jalgebra.2005.08.013},
}

@article {SwansonPrimaryDecompositions,
    AUTHOR = {Swanson, Irena},
     TITLE = {Powers of ideals. {P}rimary decompositions, {A}rtin-{R}ees
              lemma and regularity},
   JOURNAL = {Math. Ann.},
  FJOURNAL = {Mathematische Annalen},
    VOLUME = {307},
      YEAR = {1997},
    NUMBER = {2},
     PAGES = {299--313},
      ISSN = {0025-5831,1432-1807},
   MRCLASS = {13A15 (13C99)},
  MRNUMBER = {1428875},
MRREVIEWER = {L.\ J.\ Ratliff, Jr.},
       DOI = {10.1007/s002080050035},
       URL = {https://doi.org/10.1007/s002080050035},
}

@book {BrunsHerzog,
    AUTHOR = {Bruns, Winfried and Herzog, J\"{u}rgen},
     TITLE = {Cohen-{M}acaulay rings},
    SERIES = {Cambridge Studies in Advanced Mathematics},
    VOLUME = {39},
 PUBLISHER = {Cambridge University Press, Cambridge},
      YEAR = {1993},
     PAGES = {xii+403},
      ISBN = {0-521-41068-1},
   MRCLASS = {13H10 (13-02)},
  MRNUMBER = {1251956},
MRREVIEWER = {Matthew\ Miller},
}

@article {Chevalley,
    AUTHOR = {Chevalley, Claude},
     TITLE = {On the theory of local rings},
   JOURNAL = {Ann. of Math. (2)},
  FJOURNAL = {Annals of Mathematics. Second Series},
    VOLUME = {44},
      YEAR = {1943},
     PAGES = {690--708},
      ISSN = {0003-486X},
   MRCLASS = {09.1X},
  MRNUMBER = {9603},
MRREVIEWER = {R.\ J.\ Walker},
       DOI = {10.2307/1969105},
       URL = {https://doi.org/10.2307/1969105},
}

@incollection {TakagiYoshida,
    AUTHOR = {Takagi, Shunsuke and Yoshida, Ken-ichi},
     TITLE = {Generalized test ideals and symbolic powers},
      NOTE = {Special volume in honor of Melvin Hochster},
   JOURNAL = {Michigan Math. J.},
  FJOURNAL = {Michigan Mathematical Journal},
    VOLUME = {57},
      YEAR = {2008},
     PAGES = {711--724},
      ISSN = {0026-2285,1945-2365},
   MRCLASS = {13A35 (13A30 13F40)},
  MRNUMBER = {2492477},
MRREVIEWER = {Florian\ Enescu},
       DOI = {10.1307/mmj/1220879433},
       URL = {https://doi.org/10.1307/mmj/1220879433},
}

@article {hartshorneAffineDuality,
    AUTHOR = {Hartshorne, Robin},
     TITLE = {Affine duality and cofiniteness},
   JOURNAL = {Invent. Math.},
  FJOURNAL = {Inventiones Mathematicae},
    VOLUME = {9},
      YEAR = {1969/70},
     PAGES = {145--164},
      ISSN = {0020-9910,1432-1297},
   MRCLASS = {14.55},
  MRNUMBER = {257096},
MRREVIEWER = {D.\ S.\ Rim},
       DOI = {10.1007/BF01404554},
       URL = {https://doi.org/10.1007/BF01404554},
}

@article {SwansonLinear,
    AUTHOR = {Swanson, Irena},
     TITLE = {Linear equivalence of ideal topologies},
   JOURNAL = {Math. Z.},
  FJOURNAL = {Mathematische Zeitschrift},
    VOLUME = {234},
      YEAR = {2000},
    NUMBER = {4},
     PAGES = {755--775},
      ISSN = {0025-5874,1432-1823},
   MRCLASS = {13J10 (13A30)},
  MRNUMBER = {1778408},
MRREVIEWER = {D.-M.\ Popescu},
       DOI = {10.1007/s002090050007},
       URL = {https://doi.org/10.1007/s002090050007},
}

@article {McAdam,
    AUTHOR = {McAdam, Stephen},
     TITLE = {Going down},
   JOURNAL = {Duke Math. J.},
  FJOURNAL = {Duke Mathematical Journal},
    VOLUME = {39},
      YEAR = {1972},
     PAGES = {633--636},
      ISSN = {0012-7094,1547-7398},
   MRCLASS = {13G05},
  MRNUMBER = {311658},
MRREVIEWER = {Robert\ Gilmer},
       URL = {http://projecteuclid.org/euclid.dmj/1077380568},
}

@article {RG,
    AUTHOR = {R. G., Rebecca},
     TITLE = {Closure operations that induce big {C}ohen-{M}acaulay
              algebras},
   JOURNAL = {J. Pure Appl. Algebra},
  FJOURNAL = {Journal of Pure and Applied Algebra},
    VOLUME = {222},
      YEAR = {2018},
    NUMBER = {7},
     PAGES = {1878--1897},
      ISSN = {0022-4049,1873-1376},
   MRCLASS = {13D22 (13A35 13C14)},
  MRNUMBER = {3763288},
MRREVIEWER = {Geoffrey\ D.\ Dietz},
       DOI = {10.1016/j.jpaa.2017.08.011},
       URL = {https://doi.org/10.1016/j.jpaa.2017.08.011},
}

@article {Dietz,
    AUTHOR = {Dietz, Geoffrey D.},
     TITLE = {A characterization of closure operations that induce big
              {C}ohen-{M}acaulay modules},
   JOURNAL = {Proc. Amer. Math. Soc.},
  FJOURNAL = {Proceedings of the American Mathematical Society},
    VOLUME = {138},
      YEAR = {2010},
    NUMBER = {11},
     PAGES = {3849--3862},
      ISSN = {0002-9939,1088-6826},
   MRCLASS = {13C14 (13A35 13D22)},
  MRNUMBER = {2679608},
MRREVIEWER = {Florian\ Enescu},
       DOI = {10.1090/S0002-9939-2010-10417-3},
       URL = {https://doi.org/10.1090/S0002-9939-2010-10417-3},
}

@misc{BhattCM,
      title={Cohen-Macaulayness of absolute integral closures}, 
      author={Bhargav Bhatt},
      year={2021},
      eprint={2008.08070},
      archivePrefix={arXiv},
      primaryClass={math.AG}
}

@article {HHAnnals,
    AUTHOR = {Hochster, Melvin and Huneke, Craig},
     TITLE = {Infinite integral extensions and big {C}ohen-{M}acaulay
              algebras},
   JOURNAL = {Ann. of Math. (2)},
  FJOURNAL = {Annals of Mathematics. Second Series},
    VOLUME = {135},
      YEAR = {1992},
    NUMBER = {1},
     PAGES = {53--89},
      ISSN = {0003-486X,1939-8980},
   MRCLASS = {13H10 (13A02 13A35 13B22)},
  MRNUMBER = {1147957},
MRREVIEWER = {Liam\ O'Carroll},
       DOI = {10.2307/2946563},
       URL = {https://doi.org/10.2307/2946563},
}

@inproceedings{HaconLamarcheSchwede,
  title={Global generation of test ideals in mixed characteristic and applications},
  author={Christopher D. Hacon and Alicia Lamarche and Karl Schwede},
  year={2021},
  url={https://api.semanticscholar.org/CorpusID:235658414}
}

@article {HHTAMS,
    AUTHOR = {Hochster, Melvin and Huneke, Craig},
     TITLE = {{$F$}-regularity, test elements, and smooth base change},
   JOURNAL = {Trans. Amer. Math. Soc.},
  FJOURNAL = {Transactions of the American Mathematical Society},
    VOLUME = {346},
      YEAR = {1994},
    NUMBER = {1},
     PAGES = {1--62},
      ISSN = {0002-9947,1088-6850},
   MRCLASS = {13A35 (13B99 13F40)},
  MRNUMBER = {1273534},
MRREVIEWER = {Ian\ M.\ Aberbach},
       DOI = {10.2307/2154942},
       URL = {https://doi.org/10.2307/2154942},
}

@article {MaSchwedeSymbolic,
    AUTHOR = {Ma, Linquan and Schwede, Karl},
     TITLE = {Perfectoid multiplier/test ideals in regular rings and bounds
              on symbolic powers},
   JOURNAL = {Invent. Math.},
  FJOURNAL = {Inventiones Mathematicae},
    VOLUME = {214},
      YEAR = {2018},
    NUMBER = {2},
     PAGES = {913--955},
      ISSN = {0020-9910,1432-1297},
   MRCLASS = {13A35 (14F18)},
  MRNUMBER = {3867632},
MRREVIEWER = {Ana\ Bravo},
       DOI = {10.1007/s00222-018-0813-1},
       URL = {https://doi.org/10.1007/s00222-018-0813-1},
}

@misc{MurayamaSymbolic,
      title={Uniform bounds on symbolic powers in regular rings}, 
      author={Takumi Murayama},
      year={2023},
      eprint={2111.06049},
      archivePrefix={arXiv},
      primaryClass={math.AC}
}

@article {ELS,
    AUTHOR = {Ein, Lawrence and Lazarsfeld, Robert and Smith, Karen E.},
     TITLE = {Uniform bounds and symbolic powers on smooth varieties},
   JOURNAL = {Invent. Math.},
  FJOURNAL = {Inventiones Mathematicae},
    VOLUME = {144},
      YEAR = {2001},
    NUMBER = {2},
     PAGES = {241--252},
      ISSN = {0020-9910,1432-1297},
   MRCLASS = {13A10 (13H05 14Q20)},
  MRNUMBER = {1826369},
MRREVIEWER = {Irena\ Swanson},
       DOI = {10.1007/s002220100121},
       URL = {https://doi.org/10.1007/s002220100121},
}

@article {HKVErratum,
    AUTHOR = {Huneke, Craig and Katz, Daniel and Validashti, Javid},
     TITLE = {Corrigendum to ``{U}niform symbolic topologies and finite
              extensions'' [{J}. {P}ure {A}ppl. {A}lgebra 219 (2015)
              543--550]},
   JOURNAL = {J. Pure Appl. Algebra},
  FJOURNAL = {Journal of Pure and Applied Algebra},
    VOLUME = {225},
      YEAR = {2021},
    NUMBER = {6},
     PAGES = {Paper No. 106587, 2},
      ISSN = {0022-4049,1873-1376},
   MRCLASS = {13A15 (13F20 13H15)},
  MRNUMBER = {4197292},
       DOI = {10.1016/j.jpaa.2020.106587},
       URL = {https://doi.org/10.1016/j.jpaa.2020.106587},
}

@article {HKV,
    AUTHOR = {Huneke, Craig and Katz, Daniel and Validashti, Javid},
     TITLE = {Uniform equivalence of symbolic and adic topologies},
   JOURNAL = {Illinois J. Math.},
  FJOURNAL = {Illinois Journal of Mathematics},
    VOLUME = {53},
      YEAR = {2009},
    NUMBER = {1},
     PAGES = {325--338},
      ISSN = {0019-2082,1945-6581},
   MRCLASS = {13A35 (13H10)},
  MRNUMBER = {2584949},
MRREVIEWER = {D.-M.\ Popescu},
       URL = {http://projecteuclid.org/euclid.ijm/1264170853},
}

@article {HHComparison,
    AUTHOR = {Hochster, Melvin and Huneke, Craig},
     TITLE = {Comparison of symbolic and ordinary powers of ideals},
   JOURNAL = {Invent. Math.},
  FJOURNAL = {Inventiones Mathematicae},
    VOLUME = {147},
      YEAR = {2002},
    NUMBER = {2},
     PAGES = {349--369},
      ISSN = {0020-9910,1432-1297},
   MRCLASS = {13A10 (13H05)},
  MRNUMBER = {1881923},
MRREVIEWER = {Irena\ Swanson},
       DOI = {10.1007/s002220100176},
       URL = {https://doi.org/10.1007/s002220100176},
}

@article {HunekeUniformBounds,
    AUTHOR = {Huneke, Craig},
     TITLE = {Uniform bounds in {N}oetherian rings},
   JOURNAL = {Invent. Math.},
  FJOURNAL = {Inventiones Mathematicae},
    VOLUME = {107},
      YEAR = {1992},
    NUMBER = {1},
     PAGES = {203--223},
      ISSN = {0020-9910,1432-1297},
   MRCLASS = {13E05 (13E15)},
  MRNUMBER = {1135470},
MRREVIEWER = {Rafael\ H.\ Villarreal},
       DOI = {10.1007/BF01231887},
       URL = {https://doi.org/10.1007/BF01231887},
}

@book {SwansonHuneke,
    AUTHOR = {Swanson, Irena and Huneke, Craig},
     TITLE = {Integral closure of ideals, rings, and modules},
    SERIES = {London Mathematical Society Lecture Note Series},
    VOLUME = {336},
 PUBLISHER = {Cambridge University Press, Cambridge},
      YEAR = {2006},
     PAGES = {xiv+431},
      ISBN = {978-0-521-68860-4; 0-521-68860-4},
   MRCLASS = {13B22 (13A18 13A30 13A35 13H15 14A05)},
  MRNUMBER = {2266432},
MRREVIEWER = {Liam\ O'Carroll},
}
